\documentclass{amsart}

\usepackage{amsmath,amsthm,amssymb,mathrsfs,color,enumerate,accents,amsfonts,bbold,adjustbox,etoolbox,changepage,bbm,adjustbox,witharrows,microtype}
\usepackage{thmtools}
\usepackage[margin=1in]{geometry}

\usepackage{pgfplots}
\pgfplotsset{compat=1.15}
\usepgfplotslibrary{fillbetween} 
\usepackage{tikz-cd}
\usepackage{tikz}
\usetikzlibrary{tqft,knots,decorations.markings,arrows,external,hobby,positioning}

\usepackage{graphicx,xcolor,transparent}
\graphicspath{./Pictures/}

\usepackage{hyperref}
\hypersetup{
    colorlinks=true, 
    linktoc=all,     
    linkcolor=blue,  
}

\newtheorem{thm}{Theorem}[section]
\newtheorem*{thm*}{Theorem}
\newtheorem{cor}[thm]{Corollary}
\newtheorem{lem}[thm]{Lemma}
\newtheorem{prop}[thm]{Proposition}

\theoremstyle{definition}
\newtheorem{dff}[thm]{Definition}
\newtheorem{xmp}[thm]{Example}
\newtheorem{rmk}[thm]{Remark}

\newcommand{\mbb}[1]{\mathbb{#1}}
\newcommand{\mcal}[1]{\mathcal{#1}}

\DeclareMathOperator{\Endo}{End}

\definecolor{darkgreen}{RGB}{26,190,26}

\newtheorem{mainthmintro}{Theorem}
\newtheorem{mainthm}{Theorem}

\title{Invariants from involutory Hopf algebras of 3-manifolds with embedded framed links}

\author{Nicolas Bridges}
\address{Department of Mathematics \\ 150 N University St \\ Purdue University \\ West Lafayette, IN \\ 47907}
\email{bridge18@purdue.edu}

\author{Shawn X. Cui}
\address{Department of Mathematics and Department of Physics and Astronomy \\ 150 N University St \\ Purdue University \\ West Lafayette, IN \\ 47907}
\email{cui177@purdue.edu}

\date{\today}

\begin{document}

\begin{abstract}
    We give invariants of pairs $(M,L)$ consisting of a closed connected oriented three-manifold and an (oriented) framed link $L$ embedded in $M$. This invariant generalizes the Kuperberg and Hennings-Kauffman-Radford (HKR) invariants of three-manifolds. We define Heegaard-Link diagrams which represent the pair $(M,L)$ and use the data of an involutory Hopf algebra and a representation of the Drinfeld double to construct the invariant. We show that if $L$ is the empty link, then the invariant recovers the Kuperberg invariant, and if $M$ is the three-sphere and certain particular representation is chosen, then the invariant recovers the HKR invariant. We also show that if the representation is the left regular representation of the Drinfeld double, then we recover the Kuperberg invariant of the surgery manifold $M(L)$, contributing to a new proof of the relationship between the HKR and Kuperberg invariants in the semisimple setting. To this end, we give a Heegaard diagram for $M(L)$ coming from the Heegaard-Link diagram representing the pair $(M,L)$. We also introduce a colored link invariant extending the construction and show it recovers the Reshetikhin-Turaev colored link invariant.
\end{abstract}

\maketitle

\section{Introduction}
\label{Introduction_Sec}
Hopf algebras are rich algebraic objects which can be used to define quantum invariants of 3-manifolds. Well-known invariants of this type are the Kuperberg \cite{Kup91,kuperberg1997non} invariant $Z_{\mathrm{Kup}}$ and the Hennings\footnote{Also known as the Hennings-Kauffman-Radford (HKR) invariant} \cite{hen96,kauffman1995invariants} invariant $Z_{\mathrm{HKR}}$. These are defined on Heegaard diagrams and surgery link diagrams, respectively. Kuperberg's construction uses a finite-dimensional Hopf algebra, while the Hennings invariant requires a finite-dimensional ribbon Hopf algebra. Since the Drinfeld double of a unimodular Hopf algebra has a natural structure of a ribbon Hopf algebra, then it is a natural question to ask about the relationship between these two invariants.  It was first speculated in \cite{kuperberg1997non} and then conjectured explicitly in \cite{kerler2021genealogy} that $Z_{\mathrm{Kup}}$ based on a Hopf algebra $H$ equals $Z_{\mathrm{HKR}}$ based on the Drinfeld double of the Hopf algebra \footnote{The conjecture stated did not involve framings, so it only covered the case of involutory Hopf algebras.}. A verification of the conjecture in the involutory (semisimple) case has been explored extensively \cite{CHEN_KUPPUM_SRINIVASAN_2009,chang2016newproofbfzkupzhenn2semisimple,Chang2013-ae,chenrestrictedqgroups,Kerler_2003}. Then, Chang and Cui proved an extension of the conjecture, showing the equivalence of the two invariants in the non-involutory case \cite{CHANG2019621} when  framings of 3-manifolds are involved. 

The invariants mentioned for involutory Hopf algebras are special cases of invariants arising from two extensively studied topological quantum field theories (TQFTs), namely, the Witten-Reshetikhin-Turaev invariant $Z_{\mathrm{WRT}}$ \cite{Reshetikhin1991-ib,Turaev1992-rh,Turaev2016-ai} and the Turaev-Viro-Barrett-Westbury invariant $Z_{\mathrm{TVBW}}$ \cite{Barrett:1993ab,TURAEV1992865} which are defined using modular tensor categories and spherical fusion categories, respectively. It is a classic result built on a series of works \cite{walker1991witten, roberts1995skein, turaev2006quantum,Turaev2010-dh} that \[Z_{\mathrm{TVBW}}(X;\mcal C) = Z_{\mathrm{WRT}}(X;\mcal Z(\mcal C))\] for a spherical fusion category $\mcal C$, where $\mcal Z(\mcal C)$ is the Drinfeld center of the category $\mcal C$. In particular, if $\mcal C$ is the representation category of an involutory Hopf algebra $H$, then $\mcal Z(\mcal C)$ is ribbon equivalent to the representation category of the Drinfeld double $D(H)$.

In this paper, we introduce a family of invariants,
 \[Z_{\mathrm{exK}}((M,L);H, (\rho, T))\] 
 dubbed \emph{extended Kuperberg invariants}, that generalize both the Kuperberg and the Hennings invariants in the semisimple case. The input data  consist of a finite-dimensional involutory Hopf algebra $H$, a finite-dimensional representation $\rho: D(H) \to \Endo(V)$ of $D(H)$, and a trace function $T$ on $\text{Im}(\rho)$. The invariant is defined for a pair $(M,L )$ of  a closed 3-manifold $M$ with an embedded oriented framed link $L$. We diagrammatically represent these pairs up to diffeomorphism via \emph{Heegaard-Link diagrams}, which are Heegaard diagrams with an additional set of embedded curves which represent the framed link. The construction of the invariant mimics that of Kuperberg's original construction. Namely, to each Heegaard circle, we associate an iterated multiplication/co-multiplication tensor, and contract along crossings. However, we also assign tensors associated to the representation $V$ of $D(H)$ to link components and contract tensors along the crossings between Heegaard circles and links using the representation $\rho$. The invariant is independent of the orientations of the link if the trace function $T$ is antipode-invariant.
 
 \begin{mainthmintro}\label{intro:thm1}
 With the notation stated above, $Z_{\mathrm{exK}}((M,L);H, (\rho, T))$ is an invariant of the diffeomorphism class of the pair $(M,L)$.
 \end{mainthmintro}

When the link $L$ is empty, $Z_{\mathrm{exK}}$ is trivially equal to the involutory Kuperberg invariant $Z_{\mathrm{Kup}}$. The following two theorems demonstrate the relation between the Hennings invariant $Z_{\mathrm{HKR}}$ and $Z_{\mathrm{exK}}$ for various choices of representations. 

\begin{mainthmintro}\label{intro:thm2}
Let $H$ be an involutory Hopf algebra with a nonzero two-sided integral $\mu$ and nonzero two-sided cointegral $e$ such that $\mu(e) =1$. Denote by $\rho_{\mathrm{reg}}$ the left regular representation of $D(H)$ and choose $\widetilde{T} = \mathrm{Tr}/\mathrm{dim}(H)$ as the trace function, where $\mathrm{Tr}$ is the usual trace of endomorphisms. Then for an unoriented framed link $L$ embedded in $S^3$, \[Z_{\mathrm{exK}}((S^3,L);H, (\rho_{\mathrm{reg}}, \widetilde{T})) = Z_{\mathrm{HKR}}(S^3(L);D(H)),\]
where $S^3(L)$ is the 3-manifold obtained by surgery on $L$. 
\end{mainthmintro}

\begin{mainthmintro}\label{intro:thm3}
    Let $H$ be an involutory ribbon Hopf algebra with the universal R-matrix $R = \sum R' \otimes R''$ and two-sided integral $\mu$ and two-sided cointegral $e$ such that $\mu(e) =1$. Define a representation $\rho_R:D(H)\to \Endo(H)$ of the Drinfeld double of $H$ by
    \[\rho_R(f \otimes v)(h):= \sum f(R')R''vh.\]
    Then
    \[Z_{\mathrm{exK}}((S^3,L);H, (\rho_R, \widetilde{T})) = Z_{\mathrm{HKR}}(S^3(L);H)\]
\end{mainthmintro}

When the representation is chosen to be the regular representation of $D(H)$, the following theorem shows that the invariant for the pair $(M,L)$ equals the invariant of the manifold $M(L)$ obtained from $M$ by surgery on $L$.
\begin{mainthmintro}\label{intro:thm4}
    For a pair $(M,L)$,
    \[Z_{\mathrm{exK}}((M,L);H, (\rho_{\mathrm{reg}}, \widetilde{T})) = Z_{\mathrm{exK}}((M(L),\emptyset);H,\emptyset)\]
\end{mainthmintro}

A corollary of Theorems \ref{intro:thm2} and \ref{intro:thm4} is a new proof of the relation between the involutory Kuperberg and Hennings invariant that for any 3-manifold $M$, $Z_{\mathrm{Kup}}(M;H) = Z_{\mathrm{HKR}}(M; D(H))$. 

 A feature of this paper is Theorem \ref{thm:surgerydiagram} which outlines a local algorithm to obtain a Heegaard diagram for the surgery manifold $M(L)$ from a Heegaard-Link diagram for the pair $(M,L)$. While there exist algorithms to obtain Heegaard diagrams for surgery manifolds, we are not aware of this particular algorithm in the literature. Note that while the extended Kuperberg invariant does indeed capture the data of the Kuperberg invariant of the surgery manifold, it also captures more data depending on the representations chosen. We leave to future work an exploration of the results of using various other representations.

We expect that this invariant can be generalized further to the non-semisimple case, however one may need to restrict to integral curves of a combing of the 3-manifold as opposed to arbitrary links. We leave this as a future direction.

The paper is organized as follows. In Section \ref{HopfAlgebras_Sec}, we review tensor diagrams and Hopf algebras and prove some lemmas which will be used later. In Section \ref{Topology_Sec}, we introduce Heegaard-Link diagrams and prove a complete set of moves to classify pairs $(M,L)$. We also prove Theorem \ref{thm:surgerydiagram} which gives a Heegaard diagram for the surgery manifold $M(L)$. In Section \ref{DiagramInvariant_Sec}, we introduce the extended Kuperberg bracket and prove the invariance of $Z_{\mathrm{exK}}$. In this section we also introduce a state-sum extension defined on irreducible representations. Sections \ref{Hennings_Sec} and \ref{Surgery_Sec} are dedicated to proving the relationship of $Z_{\mathrm{exK}}$ with the Hennings invariant and with the surgery manifold.

\section{Hopf Algebras}
\label{HopfAlgebras_Sec}
\subsection{Tensor Diagrams and Basic Facts}

Let $V$ be a finite dimensional vector space and $V^*$ its linear dual. A \emph{tensor diagram} in $V$ is a pair $(\mcal G, \mcal T = \{\mcal T_v\})$ where,
\begin{enumerate}
    \item $\mcal G$ is a directed graph such that at each vertex $v$, there is a local ordering on the set of incoming legs (i.e. edges) and a local ordering on the set of outgoing legs by $\{1,\dots ,i_v\}$ and $\{1,\dots, o_v\}$, respectively;
    \item for each vertex $v$, $\mcal T_v\in V^1\otimes \cdots \otimes V^{i_v}\otimes V_1\otimes \cdots \otimes V_{o_v}$, where each $V^i$ is a copy of $V^*$ associated with the $i$-th incoming leg and each $V_j$ is a copy of $V$ associated with the $j$-th outgoing leg. In this case, $\mcal T_v$ is called an $(i_v, o_v)$ tensor.
\end{enumerate}

Choosing a basis $\{v_1,\dots ,v_k\}$ for $V$ and a dual basis $\{v^1,\dots ,v^k\}$ for $V^*$, then an $(m,n)$ tensor $\mcal T$ can be generically written as 
\begin{equation}\label{eq:tensor}\mcal T = \sum \mcal T_{i_1,\dots ,i_m}^{j_1,\dots j_n} v^{i_1}\otimes \cdots \otimes v^{i_m}\otimes v_{j_1}\otimes \cdots \otimes v_{j_n}\end{equation} An example of a (3,2) tensor $\mcal T$, as shown tensor-diagrammatically, is:

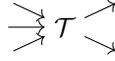
\begin{center}
    \begin{tikzpicture}
        \node at (0,0) (T) {$\mcal T$};
        \foreach \x in {155,180,-155} {
            \draw[<-] (T) -- ++(\x:0.75);
        }
        \foreach \x in {25,-25} {
            \draw[->] (T) -- ++(\x:0.75);
        }
    \end{tikzpicture}
\end{center}

In tensor diagrams, vertices are replaced by the labels of the corresponding tensors. An important convention we use is that incoming edges have local ordering given by enumerating counterclockwise. Similarly, outgoing edges are enumerated clockwise. This convention uniquely determines a local ordering if both types of edges are present. If there is any ambiguity, we will label which edge is first in the ordering.

There is a binary operation on tensors called \emph{contraction}. Given an $(m,n)$ tensor $\mcal T$ and an $(m',n')$ tensor $\mcal T'$ where $n,m'\neq 0$, then we can define a contraction of the $k$th outgoing edge of $\mcal T$ with the $l$th incoming edge of $\mcal T'$ by just connecting the corresponding edges in the diagram. In tensor notation (as in \ref{eq:tensor}), this corresponds to an identification of the $k$th lower index of $\mcal T$ and the $l$th upper index of $\mcal T'$.

For example, the figure below shows the contraction of the first outgoing edge of $\mcal T$ with the second incoming edge of $\mcal T'$, denoted $\mcal T *_{1,2}\mcal T'$:
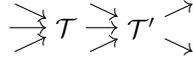
\begin{center}
    \begin{tikzpicture}
        \node at (0,0) (T) {$\mcal T$};
        \node at (1,0) (T') {$\mcal T'$};
        \foreach \x in {155,180,-155} {
            \draw[<-] (T) -- ++(\x:0.75);
        }
        \draw[->] (T) -- (T');
        \foreach \x in {155,-155} {
            \draw[<-] (T') -- ++(\x:0.75);
        }
        \foreach \x in {25,-25} {
            \draw[->] (T') -- ++(\x:0.75);
        }
    \end{tikzpicture}
\end{center}

and in tensor notation this is given by: 
\begin{align*}
    \Big (\sum \mcal T_{a,b,c}^d v^a\otimes v^b\otimes v^c \otimes v_d \Big ) &*_{1,2} \Big ( \sum \mcal{T'}_{a',b',c'}^{d',e'} v^{a'}\otimes v^{b'} \otimes v^{c'} \otimes v_{d'}\otimes v_{e'}\Big ) \\
    &= \sum \Big ( \sum_{b',d}\delta_{b'}^d\mcal T_{a,b,c}^d \mcal{T'}_{a',b',c'} ^{d',e'}\Big ) v^{a'} \otimes v^a\otimes v^b \otimes v^c \otimes v^{c'} \otimes v_{d'} \otimes v_{e'}
\end{align*}

where $\delta_x^y$ is Kronecker delta function. Notice that we ordered the tensor factors so as to match the convention we use on diagrams.

Because $V$ is finite dimensional, we may identify $(m,n)$ tensors with linear maps from $V^{\otimes m}$ to $V^{\otimes n}$. In light of this, will often use tensor contraction and function composition notation interchangeably. In particular, taking the (partial) trace of endomorphisms of a finite-dimensional vector space is realized as a tensor contraction from outgoing edges to incoming edges of the same tensor. For more discussion on tensor diagrams, see \cite{Kup91} and \cite{kuperberg1997non}.

\begin{dff}
    Let $\Bbbk$ be a field. A (finite-dimensional) \emph{Hopf algebra} over $\Bbbk$ is a tuple $(H,M,i,\Delta,\epsilon,S)$ consisting of a finite-dimensional vector space $H$ over $\Bbbk$, a (2,1) tensor $M$ called the \emph{multiplication}, a (0,1) tensor $i$ called the \emph{unit}, a (1,2) tensor $\Delta$ called the \emph{comultiplication}, a (1,0) tensor $\epsilon$ called the \emph{counit}, and a (1,1) tensor $S$ called the \emph{antipode} satisfying:

    \begin{enumerate}[\text{Axiom} 1:]
        \item (Associativity)
    \begin{center}
        \begin{tikzpicture}
            \node at (0,0) (ML) {$M$};
            \node at (1,0) (MR) {$M$};
            \draw[->] (ML) -- (MR);
            \draw[<-] (ML) -- ++(135:0.75);
            \draw[<-] (ML) -- ++(-135:0.75);
            \draw[<-] (MR) -- ++(-135:0.75);
            \draw[->] (MR) -- ++(0.75,0);
            \node at (2.5,0) {$=$};
            \begin{scope}[shift = {(3.75,0)}]
                \node at (0,0) (ML) {$M$};
                \node at (1,0) (MR) {$M$};
                \draw[->] (ML) -- (MR);
                \draw[<-] (ML) -- ++(135:0.75);
                \draw[<-] (ML) -- ++(-135:0.75);
                \draw[<-] (MR) -- ++(135:0.75);
                \draw[->] (MR) -- ++(0.75,0);
            \end{scope}
        \end{tikzpicture}
    \end{center}

    \item (Unit)
    \begin{center}
        \begin{tikzpicture}
            \node at (1,0) (M) {$M$};
            \node[above left = 0.25 of M] (i) {$i$};
            \draw[->] (i) -- (M);
            \draw[<-] (M) -- ++(-135:0.75);
            \draw[->] (M) -- ++(0.75,0);
            \node at (2.5,0) {$=$};
            \begin{scope}[shift = {(3.25,0)}]
                \node at (1,0) (M) {$M$};
                \node[below left = 0.25 of M] (i) {$i$};
                \draw[->] (i) -- (M);
                \draw[<-] (M) -- ++(135:0.75);
                \draw[->] (M) -- ++(0.75,0);
            \node at (2.5,0) {$=$};
            \begin{scope}[shift = {(3.25,0)}]
                \draw[->] (0,0) -- (1,0);
            \end{scope}
            \end{scope}
        \end{tikzpicture}
    \end{center}

    \item (Coassociativity)
    \begin{center}
        \begin{tikzpicture}
            \node at (0,0) (D1) {$\Delta$};
            \node at (1,0) (D2) {$\Delta$};
            \draw[->] (D1) -- (D2);
            \draw[<-] (D1) -- ++(-0.75,0);
            \draw[->] (D1) -- ++(45:0.75);
            \draw[->] (D2) -- ++(45:0.75);
            \draw[->] (D2) -- ++(-45:0.75);
            \node at (2.5,0) {$=$};
            \begin{scope}[shift = {(4,0)}]
                \node at (0,0) (D1) {$\Delta$};
            \node at (1,0) (D2) {$\Delta$};
            \draw[->] (D1) -- (D2);
            \draw[<-] (D1) -- ++(-0.75,0);
            \draw[->] (D1) -- ++(-45:0.75);
            \draw[->] (D2) -- ++(45:0.75);
            \draw[->] (D2) -- ++(-45:0.75);
            \end{scope}
        \end{tikzpicture}
    \end{center}

    \item (Counit)
    \begin{center}
        \begin{tikzpicture}
            \node at (0,0) (D) {$\Delta$};
            \node[above right = 0.25 of D] (e) {$\epsilon$};
            \draw[->] (D) -- (e);
            \draw[<-] (D) -- ++(-0.75,0);
            \draw[->] (D) -- ++(-45:0.75);
            \node at (1.5,0) {$=$};
            \begin{scope}[shift = {(3,0)}]
                \node at (0,0) (D) {$\Delta$};
            \node[above right = 0.25 of D] (e) {$\epsilon$};
            \draw[->] (D) -- (e);
            \draw[<-] (D) -- ++(-0.75,0);
            \draw[->] (D) -- ++(-45:0.75);
            \node at (1.5,0) {$=$};
            \begin{scope}[shift = {(3,0)}]
                \draw[->] (-0.75,0) -- (0.25,0);
            \end{scope}
            \end{scope}
        \end{tikzpicture}
    \end{center}

    \item (Bialgebra)

    \begin{center}
        \begin{tikzpicture}
            \node at (0,0) (M) {$M$};
            \node at (1,0) (D) {$\Delta$};
            \draw[->] (M) -- (D);
            \draw[<-] (M) -- ++(-135:0.75);
            \draw[<-] (M) -- ++(135:0.75);
            \draw[->] (D) -- ++(45:0.75);
            \draw[->] (D) -- ++(-45:0.75);
            \node at (2.5,0) {$=$};
            \begin{scope}[shift = {(4,0)}]
                \node at (0,0.5) (DU) {$\Delta$};
                \node at (0,-0.5) (DL) {$\Delta$};
                \node at (1,0.5) (MU) {$M$};
                \node at (1,-0.5) (ML) {$M$};
                \draw[->] (DU) -- (MU);
                \draw[->] (DU) -- (ML);
                \draw[->] (DL) -- (MU);
                \draw[->] (DL) -- (ML);
                \draw[<-] (DU) -- ++(-0.75,0);
                \draw[<-] (DL) -- ++(-0.75,0);
                \draw[->] (MU) -- ++(0.75,0);
                \draw[->] (ML) -- ++(0.75,0);
                \begin{scope}[shift = {(-4,-2)}]
                    \node at (0,0) (i) {$i$};
                    \node at (1,0) (D) {$\Delta$};
                    \draw[->] (i) -- (D);
                    \draw[->] (D) -- ++(45:0.75);
                    \draw[->] (D) -- ++(-45:0.75);
                    \node at (2.5,0) {$=$};
                    \begin{scope}[shift = {(4,0)}]
                        \node at (0,0.5) (iU) {$i$};
                        \node at (0,-0.5) (iL) {$i$};
                        \draw[->] (iU) -- ++(0.75,0);
                        \draw[->] (iL) -- ++(0.75,0);
                    \end{scope}
                \end{scope}
            \end{scope}
        \end{tikzpicture}
    \end{center}

    \item (Antipode)
    \begin{center}
        \begin{tikzpicture}
            \node at (0,0) (D) {$\Delta$};
            \node at (1,0.5) (S) {$S$};
            \node at (2,0) (M) {$M$};
            \draw[->] (D) -- (S);
            \draw[->] (S) -- (M);
            \draw[->] (D) .. controls +(0.5,-0.5) and +(-0.5,-0.5) .. (M);
            \draw[<-] (D) -- ++(-0.75,0);
            \draw[->] (M) -- ++(0.75,0);
            \node at (3.5,0) {$=$};
            \begin{scope}[shift = {(5,0)}]
                \node at (0,0) (D) {$\Delta$};
            \node at (1,-0.5) (S) {$S$};
            \node at (2,0) (M) {$M$};
            \draw[->] (D) -- (S);
            \draw[->] (S) -- (M);
            \draw[->] (D) .. controls +(0.5,0.5) and +(-0.5,0.5) .. (M);
            \draw[<-] (D) -- ++(-0.75,0);
            \draw[->] (M) -- ++(0.75,0);
            \node at (-1.5,-2) {$=$};
            \begin{scope}[shift = {(0,-2)}]
                \node at (0,0) (e) {$\epsilon$};
                \node at (1,0) (i) {$i$};
                \draw[<-] (e) -- ++(-0.75,0);
                \draw[->] (i) -- ++(0.75,0);
            \end{scope}
            \end{scope}
        \end{tikzpicture}
    \end{center}
    \end{enumerate}

    We say that $H$ is \emph{involutory} if 
    \begin{center}
        \begin{tikzpicture}
            \node at (0,0) (S1) {$S$};
            \node at (1,0) (S2) {$S$};
            \draw[->] (S1) -- (S2);
            \draw[<-] (S1) -- ++(-0.75,0);
            \draw[->] (S2) -- ++(0.75,0);
            \node at (2.5,0) {$=$};
            \draw[->] (3.25,0) -- (4.25,0);
        \end{tikzpicture}
    \end{center}
\end{dff} 

We will use the following abbreviations in light of (co)associativity.

\begin{center}
    \begin{tikzpicture}
        \node at (0,0) (M) {$M$};
        \draw[->] (M) -- ++(0.75,0);
        \foreach \angle in {90,120,150,240}
        {
            \draw[<-] (M) -- ++(\angle:0.75);
        }
        \foreach \angle in {175,195,215}
        {
            \filldraw (\angle:0.5cm) circle [radius = 0.5pt];
        }
        \node at (1.5,0) {$=$};
    \begin{scope}[shift = {(3,0)}]
        \node at (0,0) (M1) {$M$};
        \node at (1,0) (M2) {$M$};
        \node at (2,0) (dots) {$\cdots$};
        \node at (3,0) (M3) {$M$};
        \draw[<-] (M1) -- ++(-0.75,0);
        \draw[<-] (M1) -- ++(-135:0.75);
        \draw[->] (M1) -- (M2);
        \draw[<-] (M2) -- ++(-135:0.75);
        \draw[->] (M2) -- (dots);
        \draw[->] (dots) -- (M3);
        \draw[<-] (M3) -- ++(-135:0.75);
        \draw[->] (M3) -- ++(0.75,0);
    \end{scope}
    \begin{scope}[shift = {(0,-1.25)}]
        \node at (0,0) (M) {$M$};
        \draw[<-] (M) -- ++(-0.75,0);
        \draw[->] (M) -- ++(0.75,0);
        \node at (1.5,0) {$=$};
        \draw[->] (2.25,0) -- (3.25,0);
    \end{scope}
    \begin{scope}[shift = {(5,-1.25)}]
        \node at (0,0) (M) {$M$};
        \draw[->] (M) -- ++(0.75,0);
        \node at (1.5,0) {$=$};
        \node at (2.5,0) (i) {$i$};
        \draw[->] (i) -- ++(0.75,0);
    \end{scope}
    \end{tikzpicture}\vspace{0.5cm}

    \begin{tikzpicture}
        \node at (0,0) (D) {$\Delta$};
        \draw[<-] (D) -- ++(-0.75,0);
        \foreach \angle in {90,30,60,-60}
        {
            \draw[->] (D) -- ++(\angle:0.75);
        }
        \foreach \angle in {-35,-15,5}
        {
            \filldraw (\angle:0.5cm) circle [radius = 0.5pt];
        }
        \node at (1.5,0) {$=$};
    \begin{scope}[shift = {(3,0)}]
        \node at (0,0) (D1) {$\Delta$};
        \node at (1,0) (D2) {$\Delta$};
        \node at (2,0) (dots) {$\cdots$};
        \node at (3,0) (D3) {$\Delta$};
        \draw[<-] (D1) -- ++(-0.75,0);
        \draw[->] (D1) -- ++(-45:0.75);
        \draw[->] (D1) -- (D2);
        \draw[->] (D2) -- ++(-45:0.75);
        \draw[->] (D2) -- (dots);
        \draw[->] (dots) -- (D3);
        \draw[->] (D3) -- ++(-45:0.75);
        \draw[->] (D3) -- ++(0.75,0);
    \end{scope}
    \begin{scope}[shift = {(0,-1.25)}]
        \node at (0,0) (D) {$\Delta$};
        \draw[<-] (D) -- ++(-0.75,0);
        \draw[->] (D) -- ++(0.75,0);
        \node at (1.5,0) {$=$};
        \draw[->] (2.25,0) -- (3.25,0);
    \end{scope}
    \begin{scope}[shift = {(5,-1.25)}]
        \node at (0,0) (D) {$\Delta$};
        \draw[<-] (D) -- ++(-0.75,0);
        \node at (1,0) {$=$};
        \node at (2.5,0) (e) {$\epsilon$};
        \draw[<-] (e) -- ++(-0.75,0);
    \end{scope}
    \end{tikzpicture}

    \end{center}

If $H$ is a finite-dimensional Hopf algebra, then the antipode is always invertible for any field (\cite[Theorem 7.1.14]{Rad11}). Then $H^{\mathrm{op}}$ is a new Hopf algebra $(H,M^{\mathrm{op}},i,\Delta,\epsilon,S^{-1})$, where $M^{\mathrm{op}}$ is the tensor

\begin{center}
    \begin{tikzpicture}
        \node at (0,0) (M) {$M^{\mathrm{op}}$};
        \foreach \angle in {150,210} {
        \draw[<-] (M) -- ++(\angle:0.75);
        }
        \draw[->] (M) -- ++(0.75,0);
        \node at (1.5,0) {$=$};
        \begin{scope}[shift = {(3.5,0)}]
            \node at (0,0) (M) {$M$};
        \draw[<-] (M) .. controls +(-0.75,0.5) and +(0.5,0) .. ++(-155:1.25);
        \draw[<-] (M) .. controls +(-0.75,-0.5) and +(0.5,0) .. ++(155:1.25);
        \draw[->] (M) -- ++(0.75,0);
        \end{scope}
    \end{tikzpicture}
\end{center}

Similarly we can define $H^{\mathrm{cop}}$ which has opposite comultiplication, and finally we can define the \emph{dual} Hopf algebra, denoted $H^*$ which has Hopf algebra structure given by swapping $M \leftrightarrow \Delta$ and $\epsilon \leftrightarrow i$ and reversing all arrows.

\begin{lem}[\cite{Kup91}]
    The following \emph{ladders} are invertible:

    \begin{center}
        \begin{tikzpicture}
            \node at (0,0) (M) {$M$};
            \node at (0,-2) (D) {$\Delta$};
            \draw[->] (D) -- (M);
            \draw[<-] (M) -- ++(-0.75,0);
            \draw[->] (M) -- ++(0.75,0);
            \draw[<-] (D) -- ++(-0.75,0);
            \draw[->] (D) -- ++(0.75,0);
        \begin{scope}[shift = {(3,0)}]
            \node at (0,0) (M) {$M$};
            \node at (0,-2) (D) {$\Delta$};
            \node at (0,-1) (S) {$S$};
            \draw[->] (D) -- (S);
            \draw[->] (S) -- (M);
            \draw[<-] (M) -- ++(-0.75,0);
            \draw[->] (M) -- ++(0.75,0);
            \draw[<-] (D) -- ++(-0.75,0);
            \draw[->] (D) -- ++(0.75,0);
        \end{scope}
        \begin{scope}[shift = {(6,0)}]
            \node at (0,0) (M) {$M$};
            \node at (0,-2) (D) {$\Delta$};
            \draw[->] (D) -- (M);
            \draw[<-] (M) -- ++(-0.75,0);
            \draw[->] (M) -- ++(0.75,0);
            \draw[->] (D) -- ++(-0.75,0);
            \draw[<-] (D) -- ++(0.75,0);
        \end{scope}
        \begin{scope}[shift = {(9,0)}]
            \node at (0,0) (M) {$M$};
            \node at (0,-2) (D) {$\Delta$};
            \node at (0,-1) (S) {$S$};
            \draw[->] (D) -- (S);
            \draw[->] (S) -- (M);
            \draw[<-] (M) -- ++(-0.75,0);
            \draw[->] (M) -- ++(0.75,0);
            \draw[->] (D) -- ++(-0.75,0);
            \draw[<-] (D) -- ++(0.75,0);
        \end{scope}
        \end{tikzpicture}
    \end{center}
\end{lem}

\begin{lem}[\cite{Kup91}]
    The following identities hold in any Hopf algebra:
    \begin{center}
        \begin{tikzpicture}
            \node at (0,0) (M) {$M$};
            \node at (1,0) (e) {$\epsilon$};
            \draw[->] (M) -- (e);
            \draw[<-] (M) -- ++(135:0.75);
            \draw[<-] (M) -- ++(-135:0.75);
            \node at (2,0) {$=$};
            \begin{scope}[shift = {(3.5,0)}]
                \node at (0,0.5) (eU) {$\epsilon$};
                \node at (0,-0.5) (eL) {$\epsilon$};
                \draw[<-] (eU) -- ++(-0.75,0);
                \draw[<-] (eL) -- ++(-0.75,0);
            \end{scope}
            \begin{scope}[shift = {(6,0)}]
                \node at (0,0) (M) {$M$};
                \node at (1,0) (S) {$S$};
                \draw[->] (M) -- (S);
                \draw[<-] (M) -- ++(135:0.75);
                \draw[<-] (M) -- ++(-135:0.75);
                \draw[->] (S) -- ++(0.75,0);
                \node at (2.5,0) {$=$};
                \begin{scope}[shift = {(4,0)}]
                    \node at (0,0.5) (SU) {$S$};
                    \node at (0,-0.5) (SL) {$S$};
                    \node at (1,0) (M) {$M^{\mathrm{op}}$};
                    \draw[->] (SU) -- (M);
                    \draw[->] (SL) -- (M);
                    \draw[<-] (SU) -- ++(-0.75,0);
                    \draw[<-] (SL) -- ++(-0.75,0);
                    \draw[->] (M) -- ++(0.75,0);
                \end{scope}
            \end{scope}
        \end{tikzpicture}\vspace{0.5cm}

        \begin{tikzpicture}
            \node at (0,0) (i) {$i$};
                \node at (1,0) (S) {$S$};
                \draw[->] (i) -- (S);
                \draw[->] (S) -- ++(0.75,0);
                \node at (2.5,0) {$=$};
                \begin{scope}[shift = {(4,0)}]
                    \node at (0,0) (i) {$i$};
                    \draw[->] (i) -- ++(0.75,0);
                \end{scope}
        \end{tikzpicture}
    \end{center}
\end{lem}

\begin{dff}
    A \emph{left integral} $\mu_L$ in a Hopf algebra $H$ is an element of $H^*$ satisfying: 
    \begin{center}
        \begin{tikzpicture}
            \node at (0,0) (D) {$\Delta$};
            \node[below right = 0.25 of D] (mu) {$\mu_L$};
            \draw[->] (D) -- (mu);
            \draw[<-] (D) -- ++(-0.75,0);
            \draw[->] (D) -- ++(45:0.75);
            \node at (1.5,0) {$=$};
            \begin{scope}[shift = {(3,0)}]
                \node at (0,0) (mu) {$\mu_L$};
                \node at (1,0) (i) {$i$};
                \draw[<-] (mu) -- ++(-0.75,0);
                \draw[->] (i) -- ++(0.75,0);
            \end{scope}
        \end{tikzpicture}
    \end{center}

    Similarly one can define a \emph{right integral}, and \emph{left cointegral} and a \emph{right cointegral}.
    \begin{center}
        \begin{tikzpicture}
            \node at (0,0) (D) {$\Delta$};
            \node[above right = 0.25 of D] (mu) {$\mu_R$};
            \draw[->] (D) -- (mu);
            \draw[<-] (D) -- ++(-0.75,0);
            \draw[->] (D) -- ++(-45:0.75);
            \node at (1.5,0) {$=$};
            \begin{scope}[shift = {(3,0)}]
                \node at (0,0) (mu) {$\mu_R$};
                \node at (1,0) (i) {$i$};
                \draw[<-] (mu) -- ++(-0.75,0);
                \draw[->] (i) -- ++(0.75,0);
            \end{scope}
        \begin{scope}[shift = {(7,0)}]
            \node at (0,0) (M) {$M$};
            \node[below left = 0.25 of M] (e) {$e_L$};
            \draw[<-] (M) -- ++(135:0.75);
            \draw[->] (e) -- (M);
            \draw[->] (M) -- ++(0.75,0);
            \node at (1.5,0) {$=$};
            \begin{scope}[shift = {(3,0)}]
                \node at (0,0) (eps) {$\epsilon$};
                \node at (1,0) (e) {$e_L$};
                \draw[<-] (eps) -- ++(-0.75,0);
                \draw[->] (e) -- ++(0.75,0);
            \end{scope}
        \end{scope}
        \begin{scope}[shift = {(4,-2)}]
            \node at (0,0) (M) {$M$};
            \node[above left = 0.25 of M] (e) {$e_R$};
            \draw[<-] (M) -- ++(-135:0.75);
            \draw[->] (e) -- (M);
            \draw[->] (M) -- ++(0.75,0);
            \node at (1.5,0) {$=$};
            \begin{scope}[shift = {(3,0)}]
                \node at (0,0) (eps) {$\epsilon$};
                \node at (1,0) (e) {$e_R$};
                \draw[<-] (eps) -- ++(-0.75,0);
                \draw[->] (e) -- ++(0.75,0);
            \end{scope}
        \end{scope}
        \end{tikzpicture}
    \end{center}
\end{dff}

    The space of left (resp. right) (co)integrals in a Hopf algebra is one-dimensional (c.f. \cite{Kup91}) and, if the Hopf algebra is involutory, every right integral is also a left integral, in which case we call it a two-sided integral. The following lemmas include some well-known facts about integrals.

    \begin{lem}
        Let $H$ be an involutory Hopf algebra and let $\mu$ and $e$ be a two-sided integral and cointegral, respectively. Then
        \begin{center}
            \begin{tikzpicture}
                \node at (0,0) (e) {$e$};
                \node at (1,0) (D) {$\Delta$};
                \node at (2,0.5) (M) {$M^{\mathrm{op}}$};
                \draw[->] (e) -- (D);
                \draw[->] (D) -- (M);
                \draw[->] (D) -- ++(0.75,-0.5);
                \draw[<-] (M) -- ++(-0.75,0.5);
                \draw[->] (M) -- ++(0.75,0);
                \node at (3,0) {$=$};
                \begin{scope}[shift = {(3.5,0.5)}]
                \node at (0,0) (e) {$e$};
                \node at (1,0) (D) {$\Delta$};
                \node at (2,-0.5) (M) {$M$};
                \node at (1,-1) (S) {$S$};
                \draw[->] (e) -- (D);
                \draw[->] (D) -- (M);
                \draw[->] (D) -- ++(0.75,0.5);
                \draw[<-] (M) -- (S);
                \draw[<-] (S) -- ++(-0.75,0);
                \draw[->] (M) -- ++(0.75,0);
                \end{scope}
                \begin{scope}[shift = {(7.5,0)}]
                \node at (0,0) (e) {$e$};
                \node at (1,0) (D) {$\Delta$};
                \node at (2,-0.5) (M) {$M^{\mathrm{op}}$};
                \draw[->] (e) -- (D);
                \draw[->] (D) -- (M);
                \draw[->] (D) -- ++(0.75,0.5);
                \draw[<-] (M) -- ++(-0.75,-0.5);
                \draw[->] (M) -- ++(0.75,0);
                \node at (3,0) {$=$};
                \begin{scope}[shift = {(3.5,-0.5)}]
                \node at (0,0) (e) {$e$};
                \node at (1,0) (D) {$\Delta$};
                \node at (2,0.5) (M) {$M$};
                \node at (1,1) (S) {$S$};
                \draw[->] (e) -- (D);
                \draw[->] (D) -- (M);
                \draw[->] (D) -- ++(0.75,-0.5);
                \draw[<-] (M) -- (S);
                \draw[<-] (S) -- ++(-0.75,0);
                \draw[->] (M) -- ++(0.75,0);
                \end{scope}
                \end{scope}
            \end{tikzpicture}
            
            \begin{tikzpicture}
                \node at (0,0) (D) {$\Delta^{\mathrm{cop}}$};
                \node at (1,0.5) (M) {$M$};
                \node at (2,0.5) (mu) {$\mu$};
                \draw[->] (D) -- (M);
                \draw[->] (M) -- (mu);
                \draw[<-] (D) -- ++(-0.75,0);
                \draw[->] (D) -- ++(0.75,-0.5);
                \draw[<-] (M) -- ++(-0.75,0.5);
                \node at (2.5,0) {$=$};
                \begin{scope}[shift = {(3.5,0.5)}]
                \node at (0,0) (D) {$\Delta$};
                \node at (1,-0.5) (M) {$M$};
                \node at (2,-0.5) (mu) {$\mu$};
                \node at (1,0.5) (S) {$S$};
                \draw[->] (D) -- (M);
                \draw[->] (M) -- (mu);
                \draw[<-] (D) -- ++(-0.75,0);
                \draw[->] (D) -- (S);
                \draw[->] (S) -- ++(0.75,0);
                \draw[<-] (M) -- ++(-0.75,-0.5);
                \end{scope}
                \begin{scope}[shift = {(7.5,0)}]
                \node at (0,0) (D) {$\Delta^{\mathrm{cop}}$};
                \node at (1,-0.5) (M) {$M$};
                \node at (2,-0.5) (mu) {$\mu$};
                \draw[->] (D) -- (M);
                \draw[->] (M) -- (mu);
                \draw[<-] (D) -- ++(-0.75,0);
                \draw[->] (D) -- ++(0.75,0.5);
                \draw[<-] (M) -- ++(-0.75,-0.5);
                \node at (2.5,0) {$=$};
                \begin{scope}[shift = {(3.5,-0.5)}]
                \node at (0,0) (D) {$\Delta$};
                \node at (1,0.5) (M) {$M$};
                \node at (2,0.5) (mu) {$\mu$};
                \node at (1,-0.5) (S) {$S$};
                \draw[->] (D) -- (M);
                \draw[->] (M) -- (mu);
                \draw[<-] (D) -- ++(-0.75,0);
                \draw[->] (D) -- (S);
                \draw[->] (S) -- ++(0.75,0);
                \draw[<-] (M) -- ++(-0.75,0.5);
                \end{scope}
                \end{scope}
            \end{tikzpicture}
        \end{center}   
    \end{lem}

    \begin{proof}
        Each of these is easily seen to be true by just applying a ladder to both sides and using the bialgebra axiom.
    \end{proof}

    \begin{lem}\label{lem:integralequalities}
        Let $H$ be an involutory Hopf algebra and let $\mu$ and $e$ be a two-sided integral and cointegral, respectively such that $\mu(e) = 1$. Then
            \begin{center}
                \begin{tikzpicture}
                    \node at (0,0) (S) {$S$};
                    \draw[<-] (S) -- ++(-0.75,0);
                    \draw[->] (S) -- ++(0.75,0);
                    \node at (1.5,0) {$=$};
                    \node at (2.5,0) (e) {$e$};
                    \node at (3.5,0) (D) {$\Delta$};
                    \node at (4.5,0) (M) {$M$};
                    \node at (5.5,0) (mu) {$\mu$};
                    \draw[->] (e) -- (D);
                    \draw[->] (D) -- (M);
                    \draw[->] (M) -- (mu);
                    \draw[<-] (M) .. controls +(-0.5,-0.5) and +(0.5,0) .. ++(-1.5,-0.5);
                    \draw[->] (D) .. controls +(0.5,-0.5) and +(-0.5,0) .. ++(1.5,-0.5);
                \end{tikzpicture}\vspace{0.25cm}

                \begin{tikzpicture}
                    \node at (0,0) (M) {$M$};
                    \node at (1,0) (mu) {$\mu$};
                    \draw[<-] (M) -- ++(135:0.75);
                    \draw[<-] (M) -- ++(-135:0.75);
                    \draw[->] (M) -- (mu);
                    \node at (1.75,0) {$=$};
                    \node at (3,0) (M2) {$M^{\mathrm{op}}$};
                    \node at (4,0) (mu2) {$\mu$};
                    \draw[<-] (M2) -- ++(135:0.75);
                    \draw[<-] (M2) -- ++(-135:0.75);
                    \draw[->] (M2) -- (mu2);
                    \begin{scope}[shift = {(5.5,0)}]
                        \node at (0,0) (e) {$e$};
                        \node at (1,0) (D) {$\Delta$};
                        \draw[->] (e) -- (D);
                        \draw[->] (D) -- ++(45:0.75);
                        \draw[->] (D) -- ++(-45:0.75);
                        \node at (2.25,0) {$=$};
                        \node at (3,0) (e) {$e$};
                        \node at (4,0) (D) {$\Delta^{\mathrm{cop}}$};
                        \draw[->] (e) -- (D);
                        \draw[->] (D) -- ++(45:0.75);
                        \draw[->] (D) -- ++(-45:0.75);
                    \end{scope}
                \end{tikzpicture}\vspace{0.25cm}

                \begin{tikzpicture}
                    \node at (0,0) (S) {$S$};
                    \node at (1,0) (mu) {$\mu$};
                    \draw[<-] (S) -- ++(-0.75,0);
                    \draw[->] (S) -- (mu);
                    \node at (1.75,0) {$=$};
                    \node at (3,0) (mu2) {$\mu$};
                    \draw[<-] (mu2) -- ++(-0.75,0);
                    \begin{scope}[shift = {(5,0)}]
                        \node at (0,0) (e) {$e$};
                        \node at (1,0) (S) {$S$};
                        \draw[->] (e) -- (S);
                        \draw[->] (S) -- ++(0.75,0);
                        \node at (2.25,0) {$=$};
                        \node at (3,0) (e) {$e$}; 
                        \draw[->] (e) -- ++(0.75,0);
                    \end{scope}
                \end{tikzpicture}
            \end{center}
    \end{lem}

    \begin{proof}
        The first equality is easy to see by using the previous lemma.

        Applying $\mu$ or $e$ to both sides of the first equality shows that $\mu$ and $e$ are invariant under the antipode. This implies the middle equalities.
    \end{proof}

    There is a special cointegral and integral in an involutory Hopf algebra which are always nonzero. This (co)integral will be useful to us later.

    \begin{lem}[\cite{kuperberg1997non}]
        The tensor

        \begin{center}
            \begin{tikzpicture}
                \node at (0,0) (P) {$P$};
                \draw[<-] (P) -- ++(-0.75,0);
                \draw[->] (P) -- ++(0.75,0);
                \node at (1.5,0) {$=$};
                \begin{scope}[shift = {(3,0)}]
                    \node at (0,0) (M) {$M$};
                    \node at (1,0.5) (SU) {$S$};
                    \node at (1,-0.5) (SL) {$S$};
                    \node at (2,0) (D) {$\Delta$};
                    \draw[<-] (M) -- ++(-0.75,0);
                    \draw[->] (SU) -- (M);
                    \draw[->] (M) -- (SL);
                    \draw[->] (D) -- (SU);
                    \draw[->] (SL) -- (D);
                    \draw[->] (D) -- ++(0.75,0);
                \end{scope}
            \end{tikzpicture}
        \end{center}

        is both a non-zero two-sided integral and a non-zero two-sided cointegral and has trace 1.
    \end{lem}

    \begin{lem}[\cite{kuperberg1997non}]\label{lem:uniquenessofintegrals}
        Given a right integral $\mu$ and a right cointegral $e$,

        \begin{center}
            \begin{tikzpicture}
                \node at (0,0) (mu) {$\mu$};
                \node at (1,0) (e) {$e$};
                \draw[<-] (mu) -- ++(-0.75,0);
                \draw[->] (e) -- ++(0.75,0);
                \node at (2.5,0) {$=$};
                \begin{scope}[shift = {(3.5,0)}]
                    \node at (0,0.5) (e) {$e$};
                    \node at (1,0.5) (mu) {$\mu$};
                    \node at (0.5,-0.5) (P) {$P$};
                    \draw[->] (e) -- (mu);
                    \draw[<-] (P) -- ++(-0.75,0);
                    \draw[->] (P) -- ++(0.75,0);
                \end{scope}
            \end{tikzpicture}
        \end{center}
    \end{lem}

    Note that if $e$ and $\mu$ are chosen such that $\mu(e) = 1$, then this lemma implies that $\epsilon(e)\mu(i) = \mathrm{dim}(H)$.
    
\subsection{Ribbon Hopf Algebras}

\begin{dff}
    A \emph{quasitriangular} Hopf algebra is a pair $(H,R)$, where $H$ is a finite dimensional Hopf algebra, and $R\in H\otimes H$, called a universal R-matrix, is an invertible element such that

    \begin{center}
        \begin{tikzpicture}
            \node at (0,0) (R) {$R$};
            \node at (0.5,0.2) {\tiny{$1$}};
            \node at (0,-1) (D) {$\Delta$};
            \node at (1,0) (MU) {$M$};
            \node at (1,-1) (ML) {$M$};
            \draw[->] (R) -- (MU);
            \draw[->] (R) -- (ML);
            \draw[->] (D) -- (MU);
            \draw[->] (D) -- (ML);
            \draw[<-] (D) -- ++(-0.75,0);
            \draw[->] (MU) -- ++(0.75,0);
            \draw[->] (ML) -- ++(0.75,0);
            \node at (2.5,-0.5) {$=$};
            \begin{scope}[shift = {(4,0)}]
                \node at (0,-1) (R) {$R$};
                \node at (0.25,-0.6) {\tiny{$1$}};
            \node at (0,0) (D) {$\Delta^{\mathrm{cop}}$};
            \node at (1,0) (MU) {$M$};
            \node at (1,-1) (ML) {$M$};
            \draw[->] (R) -- (MU);
            \draw[->] (R) -- (ML);
            \draw[->] (D) -- (MU);
            \draw[->] (D) -- (ML);
            \draw[<-] (D) -- ++(-0.75,0);
            \draw[->] (MU) -- ++(0.75,0);
            \draw[->] (ML) -- ++(0.75,0);
            \end{scope}
        \end{tikzpicture}\vspace{0.25cm}

        \begin{tikzpicture}
            \node at (0,0) (R) {$R$};
            \node at (0.5,0.2) {\tiny{$1$}};
            \node at (1,0) (D) {$\Delta$};
            \draw[->] (R) -- (D);
            \draw[->] (R) -- ++(-45:0.75);
            \draw[->] (D) -- ++(45:0.75);
            \draw[->] (D) -- ++(-45:0.75);
            \node at (2.5,0) {$=$};
            \begin{scope}[shift = {(3.5,0)}]
                \node at (0,0.5) (RU) {$R$};
                \node at (0.25,0.9) {\tiny{$1$}};
                \node at (0.25,-0.1) {\tiny{$1$}};
                \node at (0,-0.5) (RL) {$R$};
                \node at (1,-0.5) (M) {$M$};
                \draw[->] (RU) -- ++(45:0.75);
                \draw[->] (RU) -- (M);
                \draw[->] (RL) -- ++(45:1.5);
                \draw[->] (RL) -- (M);
                \draw[->] (M) -- ++(0.75,0);
            \end{scope}
            \begin{scope}[shift = {(7,0)}]
                \node at (0,0) (R) {$R$};
            \node at (0.25,0.4) {\tiny{$1$}};
            \node at (1,0) (D) {$\Delta$};
            \draw[->] (R) -- (D);
            \draw[->] (R) -- ++(45:0.75);
            \draw[->] (D) -- ++(45:0.75);
            \draw[->] (D) -- ++(-45:0.75);
            \node at (2.5,0) {$=$};
            \begin{scope}[shift = {(3.5,0)}]
                \node at (0,0.5) (RU) {$R$};
                \node at (0.4,0.7) {\tiny{$1$}};
                \node at (0.25,-0.1) {\tiny{$1$}};
                \node at (0,-0.5) (RL) {$R$};
                \node at (1,0.5) (M) {$M^{\mathrm{op}}$};
                \draw[->] (RU) -- ++(-45:1.5);
                \draw[->] (RU) -- (M);
                \draw[->] (RL) -- ++(-45:0.75);
                \draw[->] (RL) -- (M);
                \draw[->] (M) -- ++(0.75,0);
            \end{scope}
            \end{scope}
        \end{tikzpicture}
    \end{center}

    where the little 1 represents the first leg of the R-matrix.
\end{dff}

The following statements hold in any quasitriangular Hopf algebra:

\begin{center}
    \begin{tikzpicture}
        \node at (0,0) (RI) {$R^{-1}$};
        \node at (0.25,0.4) {\tiny{$1$}};
        \node at (1.25,0) {$=$};
        \node at (2,0) (R) {$R$};
        \node at (2.25,0.4) {\tiny{$1$}};
        \node[above right =0.25 of R] (S) {$S$};
        \draw[->] (R) -- (S);
        \draw[->] (RI) -- ++(45:0.75);
        \draw[->] (RI) -- ++(-45:0.75);
        \draw[->] (R) -- ++(-45:0.75);
        \draw[->] (S) -- ++(0.75,0);
        \node at (4,0) {$=$};
        \node at (4.75,0) (R2) {$R$};
        \node at (5,0.4) {\tiny{$1$}};
        \node[below right =0.25 of R2] (S2) {$S^{-1}$};
        \draw[->] (R2) -- (S2);
        \draw[->] (R2) -- ++(45:0.75);
        \draw[->] (S2) -- ++(0.75,0);
    \end{tikzpicture}\vspace{0.25cm}

    \begin{tikzpicture}
        \node at (0,0) (R1) {$R$};
        \node at (0.25,0.4) {\tiny{$1$}};
        \node at (1.25,0) {$=$};
        \node at (2,0) (R) {$R$};
        \node at (2.25,0.4) {\tiny{$1$}};
        \node[above right =0.25 of R] (S1) {$S$};
        \node[below right =0.25 of R] (S2) {$S$};
        \draw[->] (R) -- (S1);
        \draw[->] (R1) -- ++(45:0.75);
        \draw[->] (R1) -- ++(-45:0.75);
        \draw[->] (R) -- (S2);
        \draw[->] (S1) -- ++(0.75,0);
        \draw[->] (S2) -- ++(0.75,0);
    \begin{scope}[shift = {(6,0)}]
        \node at (0,0) (R) {$R$};
        \node at (0.25,0.4) {\tiny{$1$}};
        \node[above right =0.25 of R] (e) {$\epsilon$};
        \draw[->] (R) -- (e);
        \draw[->] (R) -- ++(-45:0.75);
        \node at (1,0) {$=$};
        \node at (1.75,0) (R2) {$R$};
        \node at (2,0.4) {\tiny{$1$}};
        \node[below right =0.25 of R2] (e) {$\epsilon$};
        \draw[->] (R2) -- (e);
        \draw[->] (R2) -- ++(45:0.75);
        \node at (3,0) {$=$};
        \node at (3.5,0) (i) {$i$};
        \draw[->] (i) -- ++(0.75,0);
    \end{scope}
    \end{tikzpicture}
\end{center}

An important element in a quasitriangular Hopf algebra is the \emph{Drinfeld element}
\begin{center}
    \begin{tikzpicture}
        \node at (-1,0) {$u$};
        \node at (-0.5,0) {$:=$};
        \node at (0,0) (R) {$R$};
        \node at (0.5,0.4) {\tiny{$1$}};
        \node at (1.25,0.5) (S) {$S$};
        \node at (2,0) (M) {$M$};
        \draw[->] (R) .. controls +(0.5,0.5) and +(-0.5,-0.5) .. (M);
        \draw[->] (R) .. controls +(0.5,-0.5) and +(-0.5,-0.5) .. (S);
        \draw[->] (S) -- (M);
        \draw[->] (M) -- ++(0.75,0);
    \end{tikzpicture}
\end{center}

The Drinfeld element $u$ is invertible and $S^2(x) = uxu^{-1}$ for any $x\in H$. Moreover, $S(u)u = uS(u)$ is central in $H$.

\begin{dff}
    A \emph{ribbon} Hopf algebra is a triple $(H,R,v)$ where $(H,R)$ is a quasitriangular Hopf algebra and $v$ is an invertible central element in $H$, called the ribbon element, satisfying
    \[ v^2 = uS(u),\qquad S(v) = v,\qquad \epsilon(v) = 1,\qquad \Delta(v) = (v\otimes v) (R_{21}R)^{-1}\]
\end{dff}

\begin{prop}\label{lem:uequalsv}
    For an involutory Hopf algebra, the Drinfeld element $u$ satisfies $u = S(u)$.
\end{prop}

\begin{proof}
    Let $e$ and $\mu$ be a nonzero two-sided cointegral and integral, respectively, such that $\mu(e) = 1$. Then the following sequence of equalities demonstrates the result:
    \begin{center}
        \begin{tikzpicture}
            \node at (0,0) (R) {$R$};
            \node at (-1,-1) (e) {$e$};
            \node at (2,0) (mu) {$\mu$};
            \node at (0.5,0.2) {\tiny{$1$}};
            \node at (0,-1) (D) {$\Delta$};
            \node at (1,0) (MU) {$M$};
            \node at (1,-1) (ML) {$M$};
            \draw[->] (R) -- (MU);
            \draw[->] (R) -- (ML);
            \draw[->] (D) -- (MU);
            \draw[->] (D) -- (ML);
            \draw[<-] (D) -- (e);
            \draw[->] (MU) -- (mu);
            \draw[->] (ML) -- ++(0.75,0);
            \node at (2.5,-0.5) {$=$};
            \begin{scope}[shift = {(4,0)}]
            \node at (0,-1) (R) {$R$};
            \node at (-1,0) (e) {$e$};
            \node at (2,0) (mu) {$\mu$};
            \node at (0.25,-0.6) {\tiny{$1$}};
            \node at (0,0) (D) {$\Delta^{\mathrm{cop}}$};
            \node at (1,0) (MU) {$M$};
            \node at (1,-1) (ML) {$M$};
            \draw[->] (R) -- (MU);
            \draw[->] (R) -- (ML);
            \draw[->] (D) -- (MU);
            \draw[->] (D) -- (ML);
            \draw[<-] (D) -- (e);
            \draw[->] (MU) -- (mu);
            \draw[->] (ML) -- ++(0.75,0);
            \end{scope}
        \end{tikzpicture}\vspace{0.25cm}

        \begin{tikzpicture}
            \node at (0,0) (R) {$R$};
            \node at (-1,-1) (e) {$e$};
            \node at (0,-1) (D) {$\Delta$};
            \node at (0.75,0) (S) {$S$};
            \node at (2,-1) (M) {$M$};
            \node at (2,0) (mu) {$\mu$};
            \node at (0.25,0.2) {\tiny{$1$}};
            \draw[->] (e) -- (D);
            \draw[->] (D) -- (M);
            \draw[->] (D) -- (mu);
            \draw[->] (R) -- (S);
            \draw[->] (M) -- ++(0.75,0);
            \draw[->] (S) .. controls +(0.5,0) and +(-0.5,0.25) .. (M);
            \draw[->] (R) .. controls +(0.5,-0.5) and +(-0.5,0.75) .. (M);
            \node at (3,-0.5) {$=$};
            \begin{scope}[shift = {(4.5,0)}]
            \node at (-1,0) (e) {$e$};
            \node at (0,-1) (R) {$R$};
            \node at (0,0) (D) {$\Delta$};
            \node at (1,-1) (S) {$S$};
            \node at (2,-1) (M) {$M$};
            \node at (0.25,-0.8) {\tiny{$1$}};
            \node at (2,0) (mu) {$\mu$};
            \draw[->] (e) -- (D);
            \draw[->] (D) -- (mu);
            \draw[->] (R) -- (S);
            \draw[->] (D) -- (M);
            \draw[->] (M) -- ++(0.75,0);
            \draw[->] (S) .. controls +(0.5,0) and +(-0.5,0) .. (M);
            \draw[->] (R) .. controls +(0.5,-0.5) and +(-0.5,-0.5) .. (M);
            \end{scope}
        \end{tikzpicture}
    \end{center}

    The left-hand side reduces to $u$ and the right-hand side reduces to $S(u)$, as desired.
\end{proof}

The proposition above shows that any involutory quasitriangular Hopf algebra is ribbon with ribbon element $v = u$.

\subsection{The Drinfeld Double and Representations}

\begin{dff}[\cite{CHANG2019621}]
    Let $H$ be a Hopf algebra with invertible antipode. Define the \emph{Drinfeld double} of $H$ to be the vector space $D(H) = H^*\otimes H$ with the following Hopf algebra structure:

    \begin{center}
        \begin{tikzpicture}
            \node at (0,0) (iD) {$i^D$};
            \draw[->] (iD) -- ++(0.75,0);
            \node at (1.5,0) {$:=$};
            \node at (2.5,0.5) (eps) {$\epsilon$};
            \node at (2.5,-0.5) (i) {$i$};
            \draw[<-] (eps) -- ++(0.75,0);
            \draw[->] (i) -- ++(0.75,0);
        \begin{scope}[shift = {(6,0)}]
            \node at (0,0) (eD) {$\epsilon^D$};
            \draw[<-] (eD) -- ++(-0.75,0);
            \node at (1,0) {$:=$};
            \node at (2.5,0.5) (i) {$i$};
            \node at (2.5,-0.5) (eps) {$\epsilon$};
            \draw[->] (i) -- ++(-0.75,0);
            \draw[<-] (eps) -- ++(-0.75,0);
        \end{scope}
        \end{tikzpicture}\vspace{0.25cm}

        \begin{tikzpicture}
            \node at (0,0) (DD) {$\Delta^D$};
            \draw[<-] (DD) -- ++(-0.75,0);
            \draw[->] (DD) -- ++(45:0.75);
            \draw[->] (DD) -- ++(-45:0.75);
            \node at (1.5,0) {$:=$};
            \node at (3,0.5) (M) {$M$};
            \node at (3,-0.5) (D) {$\Delta$};
            \draw[->] (M) -- ++(-0.75,0);
            \draw[<-] (D) -- ++(-0.75,0);
            \draw[<-] (M) -- ++(30:1.5);
            \draw[<-] (M) -- ++(-30:1.5);
            \draw[->] (D) -- ++(30:1.5);
            \draw[->] (D) -- ++(-30:1.5);
        \end{tikzpicture}\vspace{0.25cm}

        \begin{tikzpicture}
            \node at (0,0) (MD) {$M^D$};
            \draw[->] (MD) -- ++(0.75,0);
            \draw[<-] (MD) -- ++(135:0.75);
            \draw[<-] (MD) -- ++(-135:0.75);
            \node at (1.5,0) {$:=$};
            \begin{scope}[shift = {(3,0)}]
                \node at (0,0.5) (DL) {$\Delta$};
                \node at (1,-0.5) (S) {$S^{-1}$};
                \node at (2,0) (ML) {$M$};
                \node at (3.5,-0.25) (MR) {$M$};
                \node at (3.5,0.25) (DR) {$\Delta^{\mathrm{cop}}$};
                \draw[<-] (DL) -- ++(-0.75,0);
                \draw[->] (DL) -- ++(0,-1) -- (S);
                \draw[->] (DL) -- ++(2,0) -- (ML);
                \draw[->] (S) -- ++(1,0) -- (ML);
                \draw[->] (ML) -- ++(-2.25,0) -- ++(0,-0.5) -- ++(-0.5,0);
                \draw[<-] (MR) -- ++(0,-0.5) -- ++(-4.25,0);
                \draw[->] (DR) -- ++(0,0.5) -- ++(-4.25,0);
                \draw[->] (DR) -- (ML);
                \draw[->] (DL) -- ++(2.25,-1) -- ++(0.5,0) -- (MR);
                \draw[<-] (DR) -- ++(0.75,0);
                \draw[->] (MR) -- ++(0.75,0);
            \end{scope}
        \end{tikzpicture}\vspace{0.25cm}

        \begin{tikzpicture}
            \node at (0,0) (SD) {$S^D$};
            \draw[<-] (SD) -- ++(-0.75,0);
            \draw[->] (SD) -- ++(0.75,0);
            \node at (1.5,0) {$:=$};
            \begin{scope}[shift = {(3,0)}]
                \node at (0,0.25) (SUL) {$S^{-1}$};
                \node at (0,-0.25) (SBL) {$S$};
                \node at (1,-0.25) (D) {$\Delta$};
                \node at (2.5,0.25) (M) {$M$};
                \node at (1.75,-0.75) (SB) {$S^{-1}$};
                \draw[->] (SUL) -- ++(-0.75,0);
                \draw[<-] (SBL) -- ++(-0.75,0);
                \draw[->] (M) -- (SUL);
                \draw[->] (SBL) -- (D);
                \draw[->] (D) -- ++(0,-0.5) -- (SB);
                \draw[->] (SB) -- ++(0.75,0) -- (M);
                \draw[->] (D) -- ++(0,1) -- ++(1.5,0) -- (M);
                \draw[<-] (M) -- ++(0.75,0);
                \draw[->] (D) -- ++(2.25,0);
            \end{scope}
        \end{tikzpicture}
    \end{center}

    \begin{prop}[\cite{Drinfeld88}]
        The Drinfeld double $D(H)$ of a finite dimensional Hopf algebra $H$ is quasitriangular with R-matrix
        \begin{center}
            \begin{tikzpicture}
                \node at (0,0) (RD) {$R^D$};
                \draw[->] (RD) -- ++(45:0.75);
                \draw[->] (RD) -- ++(-45:0.75);
                \node at (0.25,0.4) {\tiny{$1$}};
                \node at (1.5,0) {$:=$};
                \begin{scope}[shift = {(2.5,0)}]
                    \node at (0,1) (eps) {$\epsilon$};
                    \node at (0,-1) (i) {$i$};
                    \draw[<-] (eps) -- ++(0.75,0);
                    \draw[->] (i) -- ++(0.75,0);
                    \draw[->] (0.75,-0.75) arc (-90:-270:0.75);
                \end{scope}
            \end{tikzpicture}
        \end{center}
    \end{prop}

    \begin{prop}\label{prop:doubleintegral}
        Let $H$ be an involutory Hopf algebra with $\mu$ a nonzero integral and $e$ a nonzero cointegral such that $\mu(e) = 1$. Then $\lambda^D := e\otimes \mu$ is a nonzero two-sided integral for the Drinfeld double $D(H)$, and $\ell^D := \mu\otimes e$ is a nonzero two-sided cointegral for $D(H)$ such that $(e\otimes \mu)(\mu\otimes e) = 1$.
    \end{prop}

    \begin{proof}
        We show that $e\otimes \mu$ is a left integral for $D(H)$ and $\mu\otimes e$ is a left cointegral. That they are nonzero follows from $e$ and $\mu$ being nonzero. The rest will follow.

        \begin{center}
        \begin{tikzpicture}
            \node at (0,0) (DD) {$\Delta^D$};
            \node[below right =0.25 of DD] (lambda) {$\lambda^D$};
            \draw[->] (DD) -- (lambda);
            \draw[<-] (DD) -- ++(-0.75,0);
            \draw[->] (DD) -- ++(45:0.75);
            \node at (1.5,0) {$=$};
            \begin{scope}[shift = {(3,0)}]
                \node at (0,0.5) (M) {$M$};
            \node at (0,-0.5) (D) {$\Delta$};
            \node[below right=0.75 of M] (e) {$e$};
            \node[below right =0.75 of D] (mu) {$\mu$};
            \draw[->] (M) -- ++(-0.75,0);
            \draw[<-] (D) -- ++(-0.75,0);
            \draw[<-] (M) -- (e);
            \draw[<-] (M) -- ++(30:1.5);
            \draw[->] (D) -- (mu);
            \draw[->] (D) -- ++(30:1.5);
            \node at (2,0) {$=$};
            \begin{scope}[shift = {(3.5,0)}]
                \node at (0,0.25) (e) {$e$};
                \node at (0,-0.25) (mu) {$\mu$};
                \node at (0.5,0.25) (eps) {$\epsilon$};
                \node at (0.5,-0.25) (i) {$i$};
                \draw[->] (e) -- ++(-0.75,0);
                \draw[<-] (mu) -- ++(-0.75,0);
                \draw[<-] (eps) -- ++(0.75,0);
                \draw[->] (i) -- ++(0.75,0);
            \end{scope}
            \end{scope}
        \end{tikzpicture}\vspace{0.5cm}

        \begin{tikzpicture}
            \node at (0,0) (MD) {$M^D$};
            \node[below left = 0.25 of MD] (ellD) {$\ell^D$};
            \draw[->] (ellD) -- (MD);
            \draw[<-] (MD) -- ++(135:0.75);
            \draw[->] (MD) -- ++(0.75,0);
            \node at (1.5,0) {$=$};
            \begin{scope}[shift = {(3,0)}]
                \node at (0,0.5) (DL) {$\Delta$};
                \node at (1,-0.5) (S) {$S$};
                \node at (2,0) (ML) {$M$};
                \node at (3.5,-0.25) (MR) {$M$};
                \node at (3.5,0.25) (DR) {$\Delta^{\mathrm{cop}}$};
                \node at (-1,-0.5) (mu) {$\mu$};
                \node at (-1,-0.75) (e) {$e$};
                \draw[<-] (DL) -- ++(-0.75,0);
                \draw[->] (DL) -- ++(0,-1) -- (S);
                \draw[->] (DL) -- ++(2,0) -- (ML);
                \draw[->] (S) -- ++(1,0) -- (ML);
                \draw[->] (ML) -- ++(-2.25,0) -- ++(0,-0.5) -- ++(mu);
                \draw[<-] (MR) -- ++(0,-0.5) -- ++(e);
                \draw[->] (DR) -- ++(0,0.5) -- ++(-4.25,0);
                \draw[->] (DR) -- (ML);
                \draw[->] (DL) -- ++(2.25,-1) -- ++(0.5,0) -- (MR);
                \draw[<-] (DR) -- ++(0.75,0);
                \draw[->] (MR) -- ++(0.75,0);
            \end{scope}
            \begin{scope}[shift = {(3,-2)}]
                \node at (-1.5,0) {$=$};
                \node at (0,0.5) (DL) {$\Delta$};
                \node at (1,-0.5) (S) {$S$};
                \node at (2,0) (ML) {$M$};
                \node at (3.5,-0.25) (e) {$e$};
                \node at (3.5,0.25) (DR) {$\Delta^{\mathrm{cop}}$};
                \node at (-1,-0.5) (mu) {$\mu$};
                \draw[<-] (DL) -- ++(-0.75,0);
                \draw[->] (DL) -- ++(0,-1) -- (S);
                \draw[->] (DL) -- ++(2,0) -- (ML);
                \draw[->] (S) -- ++(1,0) -- (ML);
                \draw[->] (ML) -- ++(-2.25,0) -- ++(0,-0.5) -- ++(mu);
                \draw[->] (DR) -- ++(0,0.5) -- ++(-4.25,0);
                \draw[->] (DR) -- (ML);
                \draw[<-] (DR) -- ++(0.75,0);
                \draw[->] (e) -- ++(0.75,0);
            \begin{scope}[shift = {(0,-2)}]
                \node at (-1.5,0) {$=$};
                \node at (0,0.5) (DL) {$\Delta$};
                \node at (1,-0.5) (S) {$S$};
                \node at (2,0) (ML) {$M$};
                \node at (3.5,-0.25) (e) {$e$};
                \node at (3.5,0.25) (DR) {$\Delta^{\mathrm{cop}}$};
                \node at (-1,-0.5) (mu) {$\mu$};
                \draw[<-] (DL) -- ++(-0.75,0);
                \draw[->] (DL) -- ++(0,-1) -- (S);
                \draw[->] (DL) -- (ML);
                \draw[->] (S) -- ++(1,0) -- (ML);
                \draw[->] (ML) -- ++(0,0.5) -- (mu);
                \draw[->] (DR) -- ++(0,0.5) -- ++(-4.25,0);
                \draw[->] (DR) -- (ML);
                \draw[<-] (DR) -- ++(0.75,0);
                \draw[->] (e) -- ++(0.75,0);
            \begin{scope}[shift = {(0,-2)}]
                \node at (-1.5,0) {$=$};
                \node at (0,0.5) (eps) {$\epsilon$};
                \node at (1,-0.25) (e) {$e$};
                \node at (1,0.25) (D) {$\Delta^{\mathrm{cop}}$};
                \node at (-1,-0.5) (mu) {$\mu$};
                \draw[<-] (eps) -- ++(-0.75,0);
                \draw[->] (D) -- (mu);
                \draw[->] (D) -- ++(0,0.5) -- ++(-1.75,0);
                \draw[<-] (D) -- ++(0.75,0);
                \draw[->] (e) -- ++(0.75,0);
            \begin{scope}[shift = {(0,-2)}]
                \node at (-1.5,0) {$=$};
                \node at (0,-0.25) (eps) {$\epsilon$};
                \node at (0,0.25) (i) {$i$};
                \node at (1,-0.25) (e) {$e$};
                \node at (1,0.25) (mu) {$\mu$};
                \draw[->] (i) -- ++(-0.75,0);
                \draw[<-] (eps) -- ++(-0.75,0);
                \draw[->] (e) -- ++(0.75,0);
                \draw[<-] (mu) -- ++(0.75,0);
            \end{scope}
            \end{scope}
            \end{scope}
            \end{scope}
        \end{tikzpicture}
        \end{center}
    \end{proof}

    Recall that a representation $\rho:D(H)\to \Endo(V)$ of the Drinfeld double of a Hopf algebra $H$ is the data of a representation $\rho_H:H\to \Endo(V)$ and a representation $\rho_{H^*}:H^*\to \Endo(V)$ satisfying:
    \[ \rho(M^D(x\otimes y)) = \rho(x)\rho(y)\]

    Note that $H$ and $H^{*,\mathrm{cop}}$ embed naturally as Hopf algebras into $D(H)$ via $\epsilon\otimes H$ and $H^{*,\mathrm{cop}}\otimes i$, respectively. We can express the condition above more concretely in the following way: let $x\in H$ and $f\in H^*$, then
    \[\rho(fx) = \rho_{H^*}(f)\rho_{H}(x)\]
    \[\rho(xf) = \rho_H(x)\rho_{H^*}(f) = \sum \rho_{H^*}(f')\rho_{H}(x')\]
    where $\sum f'\otimes x'$ is the element of $H^*\otimes H$ given by:

    \begin{center}
        \begin{tikzpicture}
            \node at (-1,0.5) (x) {$x$};
            \node at (-1,-0.5) (f) {$f$};
            \node at (0,0.5) (DL) {$\Delta$};
                \node at (1,-0.5) (S) {$S^{-1}$};
                \node at (2,0) (ML) {$M$};
                \draw[<-] (DL) -- ++(-0.75,0);
                \draw[->] (DL) -- ++(0,-1) -- (S);
                \draw[->] (DL) -- ++(2,0) -- (ML);
                \draw[->] (S) -- ++(1,0) -- (ML);
                \draw[->] (ML) -- ++(-2.25,0) -- ++(0,-0.5) -- ++(-0.5,0);
                \draw[<-] (ML) -- ++(0.75,0);
                \draw[->] (DL) -- ++(2.25,-1) -- ++(0.5,0);
        \end{tikzpicture}
    \end{center}

    We introduce a tensor-like notation for representations which will come in handy throughout this paper. We will notate a representation $\rho$ in the following way:

    \begin{center}
        \begin{tikzpicture}
            \node[darkgreen] at (0,0) (rho) {\fbox{$\rho$}};
            \draw[<-] (rho) -- ++(-0.75,0);
            \draw[->,darkgreen] (rho) -- ++(0.75,0);
        \end{tikzpicture}
    \end{center}

    indicating that a representation has input an element of an algebra, and outputs an endomorphism. We will also use a green $\textcolor{darkgreen}{M}$ to denote multiplication in $\Endo(V)$. Any other operations or functions on $\Endo(V)$ will also be colored green for consistency. The above equalities relating representations of $H$ and $H^*$ to $D(H)$ are thus given as follows:

    \begin{center}
        \begin{tikzpicture}
            \node at (0,0) (MD) {$M^D$};
            \node[darkgreen] at (1,0) (rho) {\fbox{$\rho$}};
            \draw[<-] (MD) -- ++(135:0.75);
            \draw[<-] (MD) -- ++(-135:0.75);
            \draw[->] (MD) -- (rho);
            \draw[->, darkgreen] (rho) -- ++(0.75,0);
            \node at (2.5,0) {$=$};
            \node at (2.5,-2) {$=$};
            \begin{scope}[shift = {(4,0)}]
                \node[darkgreen] at (0,0.25) (rhoU) {\fbox{$\rho$}};
                \node[darkgreen] at (0,-0.25) (rhoL) {\fbox{$\rho$}};
                \node[darkgreen] at (1,0) (M) {$M$};
                \draw[<-] (rhoU) -- ++(-0.75,0);
                \draw[<-] (rhoL) -- ++(-0.75,0);
                \draw[->,darkgreen] (rhoU) -- (M);
                \draw[->,darkgreen] (rhoL) -- (M);
                \draw[->,darkgreen] (M) -- ++(0.75,0);
            \end{scope}
            \begin{scope}[shift = {(4,-2)}]
                \node at (0,0.5) (DL) {$\Delta$};
                \node at (1,-0.5) (S) {$S^{-1}$};
                \node at (2,0) (ML) {$M$};
                \node at (3.5,-0.25) (MR) {$M$};
                \node at (3.5,0.25) (DR) {$\Delta^{\mathrm{cop}}$};
                \node[darkgreen] at (4.75,0.25) (rhoHdual) {\fbox{$\rho_{H^*}$}};
                \node[darkgreen] at (4.75,-0.25) (rhoH) {\fbox{$\rho_H$}};
                \node[darkgreen] at (6,0) (MRR) {$M$};
                \draw[<-] (DL) -- ++(-0.75,0);
                \draw[->] (DL) -- ++(0,-1) -- (S);
                \draw[->] (DL) -- ++(2,0) -- (ML);
                \draw[->] (S) -- ++(1,0) -- (ML);
                \draw[->] (ML) -- ++(-2.25,0) -- ++(0,-0.5) -- ++(-0.5,0);
                \draw[<-] (MR) -- ++(0,-0.5) -- ++(-4.25,0);
                \draw[->] (DR) -- ++(0,0.5) -- ++(-4.25,0);
                \draw[->] (DR) -- (ML);
                \draw[->] (DL) -- ++(2.25,-1) -- ++(0.5,0) -- (MR);
                \draw[<-] (DR) -- (rhoHdual);
                \draw[->] (MR) -- (rhoH);
                \draw[->,darkgreen] (rhoHdual) -- (MRR);
                \draw[->,darkgreen] (rhoH) -- (MRR);
                \draw[->,darkgreen] (MRR) -- ++(0.75,0);
            \end{scope}
        \end{tikzpicture}
    \end{center}

    \begin{lem}\label{lem:doubleequalities}
        The following relations hold if and only if the representations $\rho_H$ and $\rho_{H^*}$ assemble into a representation of the Drinfeld double of an involutory Hopf algebra $H$:

        \begin{center}
            \begin{tikzpicture}
                \node at (0,0.5) (D) {$\Delta$};
                \node at (0,-0.5) (M) {$M^{\mathrm{op}}$};
                \node[darkgreen] at (2.25,0.5) (rhoH) {\fbox{$\rho_H$}};
                \node[darkgreen] at (2.25,-0.5) (rhoHdual) {\fbox{$\rho_{H^*}$}};
                \node[darkgreen] at (3.5,0) (MR) {$M$};
                \draw[<-] (D) -- ++(-0.75,0);
                \draw[->] (M) -- ++(-0.75,0);
                \draw[->] (D) -- (M);
                \draw[->] (D) -- (rhoH);
                \draw[->] (rhoHdual) -- (M);
                \draw[->,darkgreen] (rhoH) -- (MR);
                \draw[->,darkgreen] (rhoHdual) -- (MR);
                \draw[->,darkgreen] (MR) -- ++(0.75,0);
                \node at (5,0) {$=$};
                \begin{scope}[shift = {(6.5,0)}]
                    \node at (0,0.5) (D) {$\Delta^{\mathrm{cop}}$};
                \node at (0,-0.5) (M) {$M$};
                \node[darkgreen] at (2.25,0.5) (rhoH) {\fbox{$\rho_H$}};
                \node[darkgreen] at (2.25,-0.5) (rhoHdual) {\fbox{$\rho_{H^*}$}};
                \node[darkgreen] at (3.5,0) (MR) {$M^{\mathrm{op}}$};
                \draw[<-] (D) -- ++(-0.75,0);
                \draw[->] (M) -- ++(-0.75,0);
                \draw[->] (D) -- (M);
                \draw[->] (D) -- (rhoH);
                \draw[->] (rhoHdual) -- (M);
                \draw[->,darkgreen] (rhoH) -- (MR);
                \draw[->,darkgreen] (rhoHdual) -- (MR);
                \draw[->,darkgreen] (MR) -- ++(0.75,0);
                \end{scope}
            \end{tikzpicture}\vspace{0.25cm}

            \begin{tikzpicture}
                \node at (0,0.5) (D) {$\Delta^{\mathrm{cop}}$};
                \node at (0,-0.5) (M) {$M$};
                \node at (1,0.5) (SU) {$S$};
                \node at (1,-0.5) (SL) {$S$};
                \node[darkgreen] at (2.25,0.5) (rhoH) {\fbox{$\rho_H$}};
                \node[darkgreen] at (2.25,-0.5) (rhoHdual) {\fbox{$\rho_{H^*}$}};
                \node[darkgreen] at (3.5,0) (MR) {$M$};
                \draw[<-] (D) -- ++(-0.75,0);
                \draw[->] (M) -- ++(-0.75,0);
                \draw[->] (D) -- (M);
                \draw[->] (D) -- (SU);
                \draw[->] (SU) -- (rhoH);
                \draw[->] (rhoHdual) -- (SL);
                \draw[->] (SL) -- (M);
                \draw[->,darkgreen] (rhoH) -- (MR);
                \draw[->,darkgreen] (rhoHdual) -- (MR);
                \draw[->,darkgreen] (MR) -- ++(0.75,0);
                \node at (5,0) {$=$};
                \begin{scope}[shift = {(6.5,0)}]
                    \node at (0,0.5) (D) {$\Delta$};
                \node at (0,-0.5) (M) {$M^{\mathrm{op}}$};
                \node at (1,0.5) (SU) {$S$};
                \node at (1,-0.5) (SL) {$S$};
                \node[darkgreen] at (2.25,0.5) (rhoH) {\fbox{$\rho_H$}};
                \node[darkgreen] at (2.25,-0.5) (rhoHdual) {\fbox{$\rho_{H^*}$}};
                \node[darkgreen] at (3.5,0) (MR) {$M^{\mathrm{op}}$};
                \draw[<-] (D) -- ++(-0.75,0);
                \draw[->] (M) -- ++(-0.75,0);
                \draw[->] (D) -- (M);
                \draw[->] (D) -- (SU);
                \draw[->] (SU) -- (rhoH);
                \draw[->] (rhoHdual) -- (SL);
                \draw[->] (SL) -- (M);
                \draw[->,darkgreen] (rhoH) -- (MR);
                \draw[->,darkgreen] (rhoHdual) -- (MR);
                \draw[->,darkgreen] (MR) -- ++(0.75,0);
                \end{scope}
            \end{tikzpicture}\vspace{0.25cm}

            \begin{tikzpicture}
                \node at (0,-0.5) (D) {$\Delta$};
                \node at (0,0.5) (M) {$M$};
                \node at (1,-0.5) (SL) {$S$};
                \node[darkgreen] at (2.25,-0.5) (rhoH) {\fbox{$\rho_H$}};
                \node[darkgreen] at (2.25,0.5) (rhoHdual) {\fbox{$\rho_{H^*}$}};
                \node[darkgreen] at (3.5,0) (MR) {$M$};
                \draw[<-] (D) -- ++(-0.75,0);
                \draw[->] (M) -- ++(-0.75,0);
                \draw[->] (D) -- (M);
                \draw[->] (D) -- (SL);
                \draw[->] (SL) -- (rhoH);
                \draw[->] (rhoHdual) -- (M);
                \draw[->,darkgreen] (rhoH) -- (MR);
                \draw[->,darkgreen] (rhoHdual) -- (MR);
                \draw[->,darkgreen] (MR) -- ++(0.75,0);
                \node at (5,0) {$=$};
                \begin{scope}[shift = {(6.5,0)}]
                    \node at (0,-0.5) (D) {$\Delta^{\mathrm{cop}}$};
                \node at (0,0.5) (M) {$M^{\mathrm{op}}$};
                \node at (1,-0.5) (SL) {$S$};
                \node[darkgreen] at (2.25,-0.5) (rhoH) {\fbox{$\rho_H$}};
                \node[darkgreen] at (2.25,0.5) (rhoHdual) {\fbox{$\rho_{H^*}$}};
                \node[darkgreen] at (3.5,0) (MR) {$M^{\mathrm{op}}$};
                \draw[<-] (D) -- ++(-0.75,0);
                \draw[->] (M) -- ++(-0.75,0);
                \draw[->] (D) -- (M);
                \draw[->] (D) -- (SL);
                \draw[->] (SL) -- (rhoH);
                \draw[->] (rhoHdual) -- (M);
                \draw[->,darkgreen] (rhoH) -- (MR);
                \draw[->,darkgreen] (rhoHdual) -- (MR);
                \draw[->,darkgreen] (MR) -- ++(0.75,0);
                \end{scope}
            \end{tikzpicture}\vspace{0.25cm}

            \begin{tikzpicture}
                \node at (0,-0.5) (D) {$\Delta^{\mathrm{cop}}$};
                \node at (0,0.5) (M) {$M^{\mathrm{op}}$};
                \node at (1,0.5) (SU) {$S$};
                \node[darkgreen] at (2.25,-0.5) (rhoH) {\fbox{$\rho_H$}};
                \node[darkgreen] at (2.25,0.5) (rhoHdual) {\fbox{$\rho_{H^*}$}};
                \node[darkgreen] at (3.5,0) (MR) {$M$};
                \draw[<-] (D) -- ++(-0.75,0);
                \draw[->] (M) -- ++(-0.75,0);
                \draw[->] (D) -- (M);
                \draw[->] (D) -- (rhoH);
                \draw[->] (rhoHdual) -- (SU);
                \draw[->] (SU) -- (M);
                \draw[->,darkgreen] (rhoH) -- (MR);
                \draw[->,darkgreen] (rhoHdual) -- (MR);
                \draw[->,darkgreen] (MR) -- ++(0.75,0);
                \node at (5,0) {$=$};
            \begin{scope}[shift = {(6.5,0)}]
                \node at (0,-0.5) (D) {$\Delta$};
                \node at (0,0.5) (M) {$M$};
                \node at (1,0.5) (SU) {$S$};
                \node[darkgreen] at (2.25,-0.5) (rhoH) {\fbox{$\rho_H$}};
                \node[darkgreen] at (2.25,0.5) (rhoHdual) {\fbox{$\rho_{H^*}$}};
                \node[darkgreen] at (3.5,0) (MR) {$M^{\mathrm{op}}$};
                \draw[<-] (D) -- ++(-0.75,0);
                \draw[->] (M) -- ++(-0.75,0);
                \draw[->] (D) -- (M);
                \draw[->] (D) -- (rhoH);
                \draw[->] (rhoHdual) -- (SU);
                \draw[->] (SU) -- (M);
                \draw[->,darkgreen] (rhoH) -- (MR);
                \draw[->,darkgreen] (rhoHdual) -- (MR);
                \draw[->,darkgreen] (MR) -- ++(0.75,0);
                \end{scope}
            \end{tikzpicture}
        \end{center}
    \end{lem}

    \begin{proof}
        For each of these, put a ladder on each side so that either the left or the right is $\rho_H(\bullet)\rho_{H^*}(\bullet)$. The result will follow after simplifying.
    \end{proof}
\end{dff}

\begin{lem}\label{lem:traceisintegral}
    Let $H$ be an involutory Hopf algebra $H$ and let $e$ and $\mu$ be a non-zero cointegral and non-zero integral, respectively, such that $\mu(e) = 1$. Let $\rho_{\mathrm{reg}}:D(H)\to \mathrm{End}(D(H))$ be the left regular representation of $D(H)$. Let $\widetilde{T} = \mathrm{Tr}/\mathrm{dim}(H)$ where $\mathrm{Tr}$ is the usual trace of endomorphisms. Then 

    \begin{center}
        \begin{tikzpicture}
            \node[darkgreen] at (0,0) (rho) {\fbox{$\rho_{\mathrm{reg}}$}};
            \node[darkgreen] at (1,0) (T) {$\widetilde{T}$};
            \draw[<-] ([shift={(-0.5,-0.1)}]rho.center) -- ++(-0.5,0);
            \draw[->] ([shift={(-0.5,0.1)}]rho.center) -- ++(-0.5,0);
            \draw[->,darkgreen] (rho) -- (T);
            \node at (1.5,0) {$=$};
            \begin{scope}[shift = {(3,0)}]
                \node at (0,0.25) (e) {$e$};
                \node at (0,-0.25) (mu) {$\mu$};
                \draw[->] (e) -- ++(-0.75,0);
                \draw[<-] (mu) -- ++(-0.75,0);
            \end{scope}
        \end{tikzpicture}
    \end{center}
\end{lem}

\begin{proof}
Taking the trace of the left regular representation and computing yields

    \begin{center}
        \begin{tikzpicture}
            \node at (0,0.5) (DL) {$\Delta$};
                \node at (1,-0.5) (S) {$S^{-1}$};
                \node at (2,0) (ML) {$M$};
                \node at (3.5,-0.25) (MR) {$M$};
                \node at (3.5,0.25) (DR) {$\Delta^{\mathrm{cop}}$};
                \draw[<-] (DL) -- ++(-0.75,0);
                \draw[->] (DL) -- ++(0,-1) -- (S);
                \draw[->] (DL) -- ++(2,0) -- (ML);
                \draw[->] (S) -- ++(1,0) -- (ML);
                \draw (ML) -- ++(-2.25,0) -- ++(0,-0.5) -- ++(-0.75,0);
                \draw[<-] (MR) -- ++(0,-0.5) -- ++(-4.25,0);
                \draw[->] (DR) -- ++(0,0.5) -- ++(-4.25,0);
                \draw[->] (DR) -- (ML);
                \draw[->] (DL) -- ++(2.25,-1) -- ++(0.5,0) -- (MR);
                \draw[<-] (DR) -- (4.5,0.25) -- (4.5,-1.5) -- (-1,-1.5) -- (-1,-0.5);
                \draw (MR) -- (4,-0.25) -- (4,-1) -- (-0.75,-1) -- (-0.75,-0.75);
                
                \node at (5.5,0) {$=$};

                \begin{scope}[shift = {(7.5,0)}]
                    \node at (0,0.5) (DL) {$\Delta$};
                \node at (1,-0.5) (S) {$S^{-1}$};
                \node at (2,0) (ML) {$M$};
                \node at (3,-0.5) (P) {$P$};
                \node at (3.75,-0.5) (eps) {$\epsilon$};
                \node at (3.5,0.25) (DR) {$\Delta^{\mathrm{cop}}$};
                \draw[<-] (DL) -- ++(-0.75,0);
                \draw[->] (DL) -- ++(0,-1) -- (S);
                \draw[->] (DL) -- ++(2,0) -- (ML);
                \draw[->] (S) -- ++(1,0) -- (ML);
                \draw (ML) -- ++(-2.25,0) -- ++(0,-0.5) -- ++(-0.75,0);
                \draw[->] (DR) -- ++(0,0.5) -- ++(-4.25,0);
                \draw[->] (DR) -- (ML);
                \draw[->] (DL) -- ++(2.25,-1) -- (P);
                \draw[<-] (DR) -- (4.5,0.25) -- (4.5,-1.5) -- (-1,-1.5) -- (-1,-0.5);
                \draw[->] (P) -- (eps);
                \end{scope}

                \node at (5.5,-3) {$=$};

                \begin{scope}[shift = {(7.5,-3)}]
                \node at (0,0.5) (P) {$P$};
                \node at (0.75,0.5) (eps) {$\epsilon$};
                \node at (3.5,0.25) (DR) {$\Delta^{\mathrm{cop}}$};
                \draw[<-] (P) -- ++(-0.75,0);
                \draw[->] (DR) -- ++(0,0.5) -- ++(-4.25,0);
                \draw[<-] (DR) -- (4.5,0.25) -- (4.5,-1.5) -- (-1,-1.5) -- (-1,-0.5) -- (DR);
                \draw[->] (P) -- (eps);
                \end{scope}

                \node at (5.5,-5.5) {$=$};

                \begin{scope}[shift = {(7.5,-5.5)}]
                \node at (0,0.25) (PU) {$P$};
                \node at (0,-0.25) (PL) {$P$};
                \node at (0.75,0.25) (i) {$i$};
                \node at (0.75,-0.25) (eps) {$\epsilon$};
                \draw[->] (PU) -- ++(-0.75,0);
                \draw[->] (i) -- (PU);
                \draw[<-] (PL) -- ++(-0.75,0);
                \draw[->] (PL) -- (eps);
                \end{scope}
            
        \end{tikzpicture}
    \end{center}

    and by Lemma \ref{lem:uniquenessofintegrals}, $P(i)\otimes \epsilon(P) = \mathrm{dim}(H) e \otimes \mu$, which gives the result.
\end{proof}

\section{Topological Preliminaries}
\label{Topology_Sec}
We will now lay out the topological preliminaries. Recall that the invariant we will describe has input a pair of a 3-manifold and an embedded framed link. We will define \emph{Heegaard-Link diagrams} which completely determine such a pair, and vice-versa, up to some moves. 

\begin{dff}
    Let $M$ be a closed connected oriented 3-manifold. A \emph{knot} is an embedding $K:S^1\to M$. A disjoint union of knots with $m$ connected components is called a \emph{link} with $m$ components. We will not distinguish a knot from its image. A \emph{framing} for a knot $K$ is a choice of trivialization of its normal bundle in $M$. This amounts to choosing a knot parallel to $K$. A framing for a link $L$ is a framing for each of its components.
\end{dff}

We will consider \emph{oriented} knots and links. We will denote a knot $K$ with an orientation simply by $K$.

Suppose a link $L$ is embedded in a surface $\Sigma$ in a 3-manifold $M$. Then we say $L$ is endowed with the \emph{blackboard framing} if the choice of trivialization is a nonzero section of the tangent bundle of $\Sigma$.

\begin{dff}
    Given a closed connected oriented 3-manifold $M$ and an oriented framed link $L$ in $M$, a \emph{Heegaard-Link diagram} for the pair $(M,L)$ is a quadruple $(\Sigma,\alpha,\beta,\tilde L)$, where $(\Sigma,\alpha,\beta)$ is a Heegaard diagram for $M$ and $\tilde L$ is a set of simple closed curves properly embedded in $\Sigma$ such that the pair $(\tilde M, \tilde L)$ comprised of the 3-manifold $\tilde M$ built from the Heegaard diagram and $\tilde L$ treated as a framed link (induced by the blackboard framing in the Heegaard surface $\Sigma$) in $\tilde M$ is diffeomorphic with the pair $(M,L)$ where the restriction to a tubular neighborhood of $\tilde L$ is a framed diffeomorphism.
    
    We say that two Heegaard-Link diagrams $(\Sigma,\alpha,\beta,L)$ and $(\Sigma',\alpha',\beta',L')$ are \emph{equivalent} if there is an orientation-preserving diffeomorphism $f:\Sigma\to \Sigma'$ such that $f(\alpha) = \alpha'$, $f(\beta) = \beta'$, and $f$ restricted to a tubular neighborhood of $L$ is a framed diffeomorphism.
\end{dff}

\begin{rmk}
    We pause to remark that a Heegaard-Link diagram is a generalization of an ordinary link diagram. An ordinary link diagram for an oriented framed link $L$ may be considered as embedded on a sphere, with stabilizations resolving each crossing. This exactly specifies a Heegaard-Link diagram for the pair $(S^3,L)$. A general Heegaard-Link diagram is then a projection of the link on a Heegaard surface, with crossings marked (resolved with ``bridges''). 
    
    By an abuse of notation, we will use $L$ to denote the link in the original manifold and the link in a Heegaard-link diagram, as one would commonly do for links and link diagrams.
\end{rmk}

\begin{lem}\label{embedknotinsurface_lemma}
    Let $M$ be a closed connected orientable 3-manifold. Let $L$ be a framed link embedded in $M$. Then there exists a Heegaard splitting $M = H_u \cup H_l$ such that $L$ embeds in $\Sigma = H_u\cap H_l$ and the blackboard framing of $L$ agrees with the original framing.
\end{lem}

\begin{proof}
    Project the link to some Heegaard surface $\Sigma$ so that it is immersed and transverse to the attaching circles of the 1- and 2-handles and any self-intersections are at most transverse double points. At self-intersections, stabilize the surface by introducing a canceling 1-2 handle pair and push off one of the arcs across the ``bridge'' which is formed. Perform this in such a way that the projection is isotopic to the original link. The result is a Heegaard-Link diagram for the link $L$ treated as an unframed link.

    Now, let $(\Sigma,\alpha,\beta,L)$ be some (unframed) Heegaard-Link diagram representing the pair $(M,L)$. Since the set of framings on each component $L_i$ is a torsor over $\mbb Z$, we know that the blackboard framing of $L$ in $\Sigma$ differs from the endowed framing in the hypotheses by an integer $k_i$ for each component.

    We now perform the following procedure to ensure the two framings match, i.e. we wish to describe a way to take each $k_i$ to 0. For each $L_i$, stabilize the Heegaard diagram near $L_i$ and reroute it so that it traverses both the longitude of the stabilized torus once and the meridian  $|k_i|$ times. There are two ways to do this up to equivalence which correspond to taking $k_i$ to $k_i-|k_i|$ or taking $k_i$ to $k_i+|k_i|$. One of these will take $k_i$ to 0 depending on if $k_i$ is positive or negative. After doing this for each component, the result is an embedded link in $\Sigma$ whose blackboard framing matches the endowed framing.
\end{proof}

The final results of this section will be describing an algorithm to obtain a Heegaard diagram for the manifold obtained by surgery on $M$ along an embedded framed link $L$ and listing a complete set of moves for Heegaard-Link diagrams. The surgery result is known and demonstrated in \cite{Szabo-Heegaard-Book}, but we give another proof here. Note that although the statement is for framed knots, we may extend it to links in the obvious way.

\begin{thm}\label{thm:surgerydiagram}
    Let $K$ be a framed knot embedded in a 3-manifold $M$. Let $E=(\Sigma,\alpha,\beta,K)$ be a Heegaard-Link diagram for the pair $(M,K)$ such that the blackboard framing of $K$ in $\Sigma$ agrees with its endowed framing. Then there exists a genus $g(\Sigma)+1$ Heegaard diagram for the 3-manifold $M(K)$ obtained by surgery of $M$ along $K$ which is obtained from $E$ in the following way (see Figure \ref{fig:SurgeryHeegaardDiagram}):
    \begin{enumerate}[1)]
        \item Perform an isotopy of the curve sets $\alpha$ and $\beta$ so that there is a separation of $K$ into two curves $K_1$ and $K_2$ with $|\alpha\cap K_1| = 0$ and $|\beta\cap K_2| = 0$. 
        \item Choose a tubular neighborhood $N$ of $K_1$ in $\Sigma$ and label the components of $N\setminus K_1$ $N'$ and $N''$. Remove two open disks away from $\beta$ from both $N'$ and $N''$.
        \item Attach $I\times S^1$ to $\Sigma$ by identifying the two boundaries with the two boundaries introduced in the previous step.
        \item Add a new circle $K_{\alpha}$ to $\alpha$ and $K_{\beta}$ to $\beta$ such that $K_\beta$ replaces the knot $K$, and $K_\alpha$ is the belt of the attached $I\times S^1$.
        \item Reroute the $\beta$ curves which intersected $K$ across the attached $I\times S^1$ so that it does not traverse the meridian.
    \end{enumerate}
\end{thm}

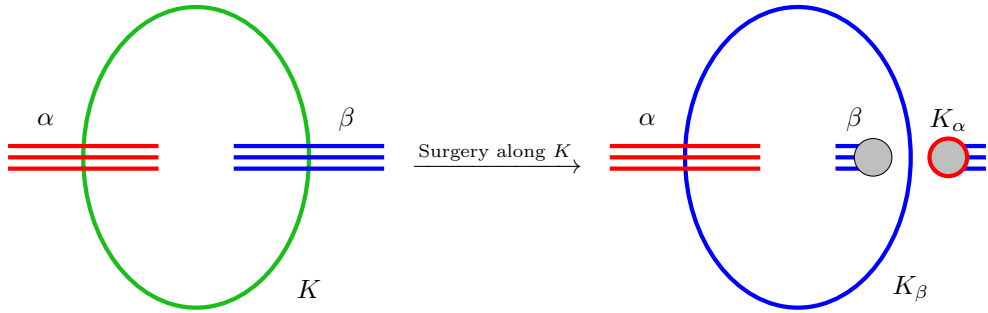
\begin{figure}
    \centering
    \begin{tikzpicture}
        \draw[darkgreen, ultra thick] (0,0) ellipse (1.5cm and 2cm);
        \draw[red, ultra thick] (-0.5,0) -- (-2.5,0);
        \draw[red, ultra thick] (-0.5,0.15) -- (-2.5,0.15);
        \draw[red, ultra thick] (-0.5,-0.15) -- (-2.5,-0.15);
        \draw[blue, ultra thick] (0.5,0) -- (2.5,0);
        \draw[blue, ultra thick] (0.5,0.15) -- (2.5,0.15);
        \draw[blue, ultra thick] (0.5,-0.15) -- (2.5,-0.15);
        \node at (1.5,-1.75) {$K$};
        \node at (-2,0.5) {$\alpha$};
        \node at (2,0.5) {$\beta$};
        \node at (4,0) {$\xrightarrow{\text{Surgery along }K}$};
    \begin{scope}[shift = {(8,0)}]
        \draw[blue, ultra thick] (0,0) ellipse (1.5cm and 2cm);
        \draw[red, ultra thick] (-0.5,0) -- (-2.5,0);
        \draw[red, ultra thick] (-0.5,0.15) -- (-2.5,0.15);
        \draw[red, ultra thick] (-0.5,-0.15) -- (-2.5,-0.15);
        \draw[blue, ultra thick] (0.5,0) -- (1,0);
        \draw[blue, ultra thick] (0.5,0.15) -- (1,0.15);
        \draw[blue, ultra thick] (0.5,-0.15) -- (1,-0.15);
        \draw[blue, ultra thick] (2,0) -- (2.5,0);
        \draw[blue, ultra thick] (2,0.15) -- (2.5,0.15);
        \draw[blue, ultra thick] (2,-0.15) -- (2.5,-0.15);
        \draw[draw=black, fill=lightgray] (1,0) circle (0.25);
        \draw[draw=black, fill=lightgray] (2,0) circle (0.25);
        \draw[red, ultra thick] (2,0) circle (0.25);
        \node at (1.5,-1.75) {$K_{\beta}$};
        \node at (-2,0.5) {$\alpha$};
        \node at (0.75,0.5) {$\beta$};
        \node at (2,0.5) {$K_\alpha$};
        \end{scope}
    \end{tikzpicture}
    \caption{A Heegaard Diagram for the surgery manifold $M(K)$}
    \label{fig:SurgeryHeegaardDiagram}
\end{figure}

\begin{proof}
    Recall surgery of $M$ along $K$:
    \begin{enumerate}[1)]
        \item Remove a tubular neighborhood $K\times D^2$ from $M$.
        \item Glue a solid torus according to the surgery coefficient of $K$ along the boundary of $M \setminus (K\times D^2)$.
    \end{enumerate}

    To construct the Heegaard diagram, first perform isotopies to the curve sets so that all $\alpha$ curves intersecting $K$ are on one side and all $\beta$ curves intersecting $K$ are on the other side. Stabilize the diagram, forming a bridge across $K$, and slide the $\beta$ curves across this bridge.

    
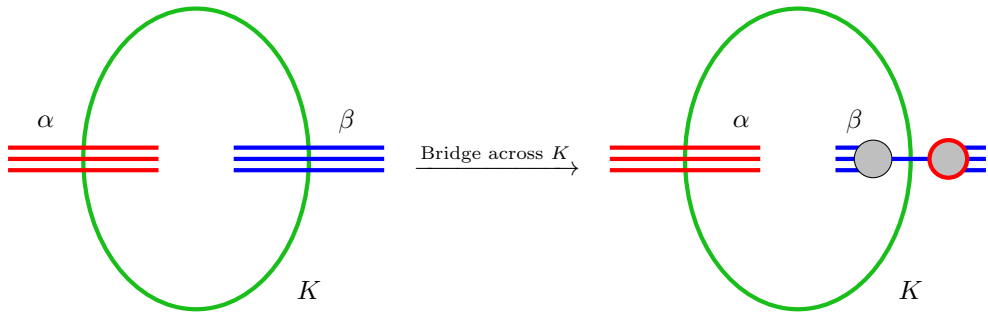
\begin{center}
        \begin{tikzpicture}
        \draw[darkgreen, ultra thick] (0,0) ellipse (1.5cm and 2cm);
        \draw[red, ultra thick] (-0.5,0) -- (-2.5,0);
        \draw[red, ultra thick] (-0.5,0.15) -- (-2.5,0.15);
        \draw[red, ultra thick] (-0.5,-0.15) -- (-2.5,-0.15);
        \draw[blue, ultra thick] (0.5,0) -- (2.5,0);
        \draw[blue, ultra thick] (0.5,0.15) -- (2.5,0.15);
        \draw[blue, ultra thick] (0.5,-0.15) -- (2.5,-0.15);
        \node at (1.5,-1.75) {$K$};
        \node at (-2,0.5) {$\alpha$};
        \node at (2,0.5) {$\beta$};
        \node at (4,0) {$\xrightarrow{\text{Bridge across }K}$};
    \begin{scope}[shift = {(8,0)}]
        \draw[darkgreen, ultra thick] (0,0) ellipse (1.5cm and 2cm);
        \draw[red, ultra thick] (-0.5,0) -- (-2.5,0);
        \draw[red, ultra thick] (-0.5,0.15) -- (-2.5,0.15);
        \draw[red, ultra thick] (-0.5,-0.15) -- (-2.5,-0.15);
        \draw[blue, ultra thick] (0.5,0) -- (1,0);
        \draw[blue, ultra thick] (0.5,0.15) -- (1,0.15);
        \draw[blue, ultra thick] (0.5,-0.15) -- (1,-0.15);
        \draw[blue, ultra thick] (2,0) -- (2.5,0);
        \draw[blue, ultra thick] (2,0.15) -- (2.5,0.15);
        \draw[blue, ultra thick] (2,-0.15) -- (2.5,-0.15);
        \draw[draw=black, fill=lightgray] (1,0) circle (0.25);
        \draw[draw=black, fill=lightgray] (2,0) circle (0.25);
        \draw[blue, ultra thick] (1.75,0) -- (1.25,0);
        \draw[red, ultra thick] (2,0) circle (0.25);
        \node at (1.5,-1.75) {$K$};
        \node at (-0.75,0.5) {$\alpha$};
        \node at (0.75,0.5) {$\beta$};
        \end{scope}
    \end{tikzpicture}
    \end{center}

    Then the following is a the relevant portion of a generalized Heegaard diagram for the manifold with boundary $M\setminus (K\times D^2)$ given by removing a tubular neighborhood of $K$ from the upper handlebody $H_\beta$ (where $K$ used to be is indicated by the dashed line):

    
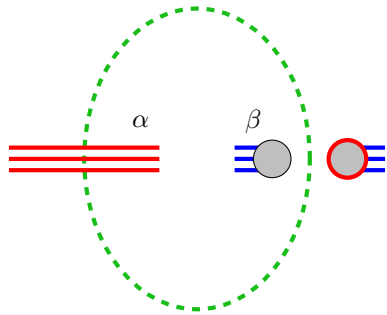
\begin{center}
    \begin{tikzpicture}
        \draw[darkgreen, ultra thick, dashed] (0,0) ellipse (1.5cm and 2cm);
        \draw[red, ultra thick] (-0.5,0) -- (-2.5,0);
        \draw[red, ultra thick] (-0.5,0.15) -- (-2.5,0.15);
        \draw[red, ultra thick] (-0.5,-0.15) -- (-2.5,-0.15);
        \draw[blue, ultra thick] (0.5,0) -- (1,0);
        \draw[blue, ultra thick] (0.5,0.15) -- (1,0.15);
        \draw[blue, ultra thick] (0.5,-0.15) -- (1,-0.15);
        \draw[blue, ultra thick] (2,0) -- (2.5,0);
        \draw[blue, ultra thick] (2,0.15) -- (2.5,0.15);
        \draw[blue, ultra thick] (2,-0.15) -- (2.5,-0.15);
        \draw[draw=black, fill=lightgray] (1,0) circle (0.25);
        \draw[draw=black, fill=lightgray] (2,0) circle (0.25);
        \draw[red, ultra thick] (2,0) circle (0.25);
        \node at (-0.75,0.5) {$\alpha$};
        \node at (0.75,0.5) {$\beta$};
    \end{tikzpicture}
    \end{center}

    Attach a solid torus along this boundary to obtain the following Heegaard diagram for the surgery manifold $M(K)$:
    
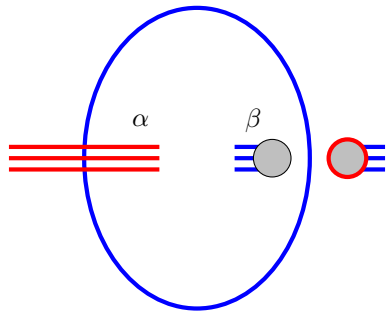
    \begin{center}
    \begin{tikzpicture}
        \draw[blue, ultra thick] (0,0) ellipse (1.5cm and 2cm);
        \draw[red, ultra thick] (-0.5,0) -- (-2.5,0);
        \draw[red, ultra thick] (-0.5,0.15) -- (-2.5,0.15);
        \draw[red, ultra thick] (-0.5,-0.15) -- (-2.5,-0.15);
        \draw[blue, ultra thick] (0.5,0) -- (1,0);
        \draw[blue, ultra thick] (0.5,0.15) -- (1,0.15);
        \draw[blue, ultra thick] (0.5,-0.15) -- (1,-0.15);
        \draw[blue, ultra thick] (2,0) -- (2.5,0);
        \draw[blue, ultra thick] (2,0.15) -- (2.5,0.15);
        \draw[blue, ultra thick] (2,-0.15) -- (2.5,-0.15);
        \draw[draw=black, fill=lightgray] (1,0) circle (0.25);
        \draw[draw=black, fill=lightgray] (2,0) circle (0.25);
        \draw[red, ultra thick] (2,0) circle (0.25);\
        \node at (-0.75,0.5) {$\alpha$};
        \node at (0.75,0.5) {$\beta$};
    \end{tikzpicture}
    \end{center}
\end{proof}

\begin{xmp}
    Consider the following genus 1 Heegaard-Link diagram for the pair $(S^1\times S^2, L)$ where $L$ is the blackboard framed oriented knot parallel to the Heegaard circles.

    \begin{center}
        \begin{tikzpicture}
        \draw[closed, ultra thick] (0,0) to[curve through = {(1.5,-1) ..  (3,0) .. (1.5,1)}] (0,0);
        \draw[ultra thick] (1.5,0) ellipse (0.65 and 0.25); 
        \draw[red, ultra thick] (1.2,-0.21) arc(90:270:0.2 and 0.38);
        \node at (1,-1.25) {$\alpha$};
        \draw[red, ultra thick, dashed] (1.2,-0.21) arc(90:-90:0.2 and 0.38);
        \draw[blue, ultra thick] (1.5,-0.23) arc(90:270:0.2 and 0.38);
        \node at (1.5,-1.25) {$\beta$};
        \draw[blue, ultra thick, dashed] (1.5,-0.23) arc(90:-90:0.2 and 0.38);
        \draw[green, ultra thick, decoration={markings, mark=at position 0.5 with {\arrow{>}}}, postaction={decorate}] (1.8,-0.21) arc(90:270:0.2 and 0.38);
        \node at (2,-1.25) {$L$};
        \draw[green, ultra thick, dashed] (1.8,-0.21) arc(90:-90:0.2 and 0.38);
    \end{tikzpicture}
    \end{center}

    Then the Heegaard diagram for the surgery manifold is given by

    \begin{center}
        \begin{tikzpicture}
        \draw[closed, ultra thick] (0,0) to[curve through = {(2,-1) .. (3.5,-0.5) .. (5,-1) .. (7,0) .. (5,1) .. (3.5,0.5) .. (2,1)}] (0,0);
        \draw[ultra thick] (1.5,0) ellipse (0.65 and 0.35);
        \draw[ultra thick] (5.5,0) ellipse (0.65 and 0.35);
        \draw[red, ultra thick] (1.2,-0.29) arc(90:270:0.2 and 0.38);
        \draw[red, ultra thick, dashed] (1.2,-0.29) arc(90:-90:0.2 and 0.38);
        \node at (0.75,-0.6) {$\alpha'_1$};
        \draw[blue, ultra thick] (1.7,-0.3) arc(90:270:0.2 and 0.38);
        \draw[blue, ultra thick, dashed] (1.7,-0.3) arc(90:-90:0.2 and 0.38);
        \node at (2.15,-0.6) {$\beta'_1$};
        \draw[red, ultra thick] (5.5,-0.33) arc(90:270:0.2 and 0.38);
        \draw[red, ultra thick, dashed] (5.5,-0.33) arc(90:-90:0.2 and 0.38);
        \node at (5.5,-1.4) {$\alpha'_2$};
        \draw[blue, ultra thick] (2.15,0) arc(180:0:1.35 and 0.2);
        \draw[blue, ultra thick, dashed] (2.15,0) arc(-180:0:1.35 and 0.2);
        \node at (2.5,0.4) {$\beta'_2$};
    \end{tikzpicture}
    \end{center}

    Note that by sliding $\beta_2'$ over $\beta_1'$ we obtain the standard Heegaard diagram for $(S^1\times S^2) \# (S^1\times S^2)$. Note that since $L$ is a 0-framed unknot in $S^1 \times S^2$, indeed performing surgery along $L$ yields the same manifold.
\end{xmp}

\begin{xmp}
    Consider the following genus 1 Heegaard-Link diagram for the pair $(S^3,L)$ where $L$ is a $(-p)$-framed unknot.
    
    \begin{center}
        \begin{tikzpicture}
        \draw[closed, ultra thick] (0,0) to[curve through = {(3,-2) ..  (6,0) .. (3,2)}] (0,0);
        \draw[ultra thick] (3,0) ellipse (1.3 and 0.5); 
        \draw[ultra thick, red] (3,0) ellipse (2 and 1);
        \draw[ultra thick, blue] (3,-0.5) arc(90:-90:0.5 and 0.75);
        \draw[ultra thick, blue, dashed] (3,-0.5) arc(90:270:0.5 and 0.75);
        \draw[ultra thick, green, decoration={markings, mark=at position 0.5 with {\arrow{>}}}, postaction={decorate}] (3,0) ellipse (2.5 and 1.25);
        \draw[ultra thick, green, fill = white] (2.25,0.25) rectangle (3.75,2.25);
        \node at (3,1.6) {$p$};
        \node at (3,1.25) {left-hand};
        \node at (3,0.9) {twists};
    \end{tikzpicture}
    \end{center}
    
    The resulting surgery manifold is well-known to be the lens space $L(p,1)$, which has a Heegaard diagram $(\mbb{T}^2,\alpha,\beta)$ where $\alpha$ is a standard meridian on $\mbb{T}^2$ and $\beta$ is a $(p,1)$ torus knot, which traverses the longitude $p$ times and the meridian $1$ time.

    The resulting Heegaard diagram for the surgery manifold is given by

    \begin{center}
        \begin{tikzpicture}
        \draw[closed, ultra thick] (0,0) to[curve through = {(3,-2) ..  (6,-1) .. (9,-2) .. (12,0) .. (9,2) .. (6,1) .. (3,2)}] (0,0);
        \draw[ultra thick] (3,0) ellipse (1.3 and 0.5); 
        \draw[ultra thick] (9,0) ellipse (1.3 and 0.5); 
        \draw[ultra thick, red] (3,0) ellipse (2 and 1);
        \draw[ultra thick, blue, decoration={markings, mark=at position 0.5 with {\arrow{>}}}, postaction={decorate}] (3,1.5) to[curve through = {(1,1) .. (0.5,0) .. (1,-1) .. (3.5,-1.5) .. (6,0)}] (7.7,0);
        \draw[ultra thick, blue, dashed] (7.7,0) arc(0:-180:1.7 and 0.5);
        \draw[ultra thick, blue] (4.3,0) to[curve through = {(4,1.25)}] (3,1.5);
        \draw[ultra thick, blue, fill = white] (2.25,0.25) rectangle (3.75,2.25);
        \node at (3,1.6) {$p$};
        \node at (3,1.25) {left-hand};
        \node at (3,0.9) {twists};
        \draw[ultra thick, red] (9,-0.5) arc(90:-90:0.5 and 0.75);
        \draw[ultra thick, red, dashed] (9,-0.5) arc(90:270:0.5 and 0.75);
        \draw[ultra thick, blue] (3,-0.5) to[curve through = {(3.25,-0.6) .. (6,0.5) .. (9,1) .. (11.5,0) .. (7,-0.75) .. (4,-1.5)}] (3,-2);
        \draw[ultra thick, blue, dashed] (3,-0.5) arc(90:270:0.5 and 0.75);

        \node at (4.75,0.9) {$\alpha'_1$};
        \node at (9.75,-1.7) {$\alpha'_2$};
        \node at (0.8,1.2) {$\beta'_1$};
        \node at (11,0) {$\beta'_2$};
    \end{tikzpicture}
    \end{center}

    Sliding $\alpha'_1$ along $\alpha'_2$ via the path defined by $\beta'_2$ yields a cancelling pair of Heegaard circles. Specifically $\alpha'_2$ cancels with $\beta'_2$. Then after destabilizing we obtain the diagram 

    \begin{center}
        \begin{tikzpicture}
        \draw[closed, ultra thick] (0,0) to[curve through = {(3,-2) ..  (6,0) .. (3,2)}] (0,0);
        \draw[ultra thick] (3,0) ellipse (1.3 and 0.5); 
        \draw[ultra thick, red] (3,0) ellipse (2 and 1);
        \draw[ultra thick, blue, decoration={markings, mark=at position 0.5 with {\arrow{>}}}, postaction={decorate}] (3,0) ellipse (2.5 and 1.25);
        \draw[ultra thick, blue, fill = white] (2.25,0.25) rectangle (3.75,2.25);
        \node at (3,1.6) {$p$};
        \node at (3,1.25) {left-hand};
        \node at (3,0.9) {twists};
    \end{tikzpicture}
    \end{center}

    which is the ``upside down'' Heegaard diagram for the Lens space $L(p,1)$.
\end{xmp}

\begin{thm}[Heegaard-Link Moves] \label{thm:extendedHeegaardMoves}
    Let $M$ be a closed connected oriented 3-manifold and let $L$ be an $m$-component oriented link in $M$. Then any two Heegaard-Link diagrams for the pair $(M,L)$ are diffeomorphic after a finite sequence of the following moves:
    \begin{enumerate}[Move 1:]
        \item (Diagram Isotopy) An isotopy of the curve sets $\alpha$ and $\beta$, or an isotopy of the link $L$.
        \item (Handle Slide) A handle slide of the Heegaard curve sets away from the link or a slide of the link across a Heegaard circle which does not intersect any link components.
        \item (Stabilization/Destabilization) A stabilization or destabilization of the Heegaard diagram at a disk away from the link $L$.
    \end{enumerate}
\end{thm}

\begin{proof}
    A version of this is proven in \cite[Section 3]{ozsvath2008holomorphic} using $2l$-pointed Heegaard diagrams. There is a relationship between these and Heegaard-Link diagrams which we do not explore. For our purposes, we prove this result differently.

    We first note that any two Heegaard diagrams for $M$ are diffeomorphic after a finite sequence of diagram isotopies, handle slides, and (de)stabilizations. Now, consider a Heegaard-Link diagram $(\Sigma,\alpha,\beta,L)$ as we examine isotopies of $L$ concentrated in the following spaces built by attaching handles to $\Sigma$. 
    
    First, we thicken $\Sigma$ to $\Sigma\times I$ and notice that any framed isotopy of $L$ in this space (followed by embedding it back to $\Sigma = \Sigma\times \{0\}$) can be realized as an isotopy of $L$ in the Heegaard surface $\Sigma$. This is because any isotopy of $\Sigma\times I$ is equivalent to an isotopy of isotopies of $\Sigma$.

    Next, we attach thickened disks to each of the $\alpha$ and $\beta$ curves. Note that any framed isotopy of $L$ in this space (followed again by embedding back to $\Sigma$) is realized as an isotopy of $L$ in $\Sigma$ with the addition of a slide across the Heegaard circles. It is important to note here that the blackboard framing agrees before and after the framed isotopy since any isotopy concentrated in these thickened disks can be modified to be an isotopy of bands in $D^2\times I$ rel $S^1\times (\{0\}\cup \{1\})$, which obviously preserves the framing.

    Finally, we attach balls to the boundary sphere components. Any isotopy of $L$ concentrated in these balls can be modified to be an isotopy of a finite number of ribbons rel the boundary. This can be realized as an isotopy on the bounding sphere, which is a composition of isotopies in $\Sigma$ and handle slides.
\end{proof}

\section{Invariants of Heegaard-Link Diagrams}
\label{DiagramInvariant_Sec}
We dedicate this section to describing the invariant of Heegaard-Link diagrams and proving the main results of the paper.

Let $H$ be a finite dimensional involutory Hopf algebra and let $\mu\in H^*$ be a nonzero two-sided integral and $e\in H$ a nonzero two-sided cointegral such that $\mu(e) = 1$. Let $\rho:D(H)\to \Endo(V)$ be a representation of the Drinfeld double of $H$. Consider also a linear map $T:\Endo(V)\to \Bbbk$ which satisfies the property that $T(XY) = T(YX)$ (i.e. $T$ is \emph{trace-like}).

Let $E = (\Sigma,\alpha,\beta,L)$ be a Heegaard-Link diagram, and $L = L_1\cup \cdots \cup L_m$, blackboard framed in $\Sigma$. Orient all Heegaard circles arbitrarily and consider an arbitrary base point on each circle and link component such that the base point does not lie on a crossing.

For each $\alpha_i$, assign the tensor 

    \begin{center}
    \begin{tikzpicture}
        \node at (0,0) (e) {$e$};
        \node at (1,0) (D) {$\Delta$};
        \node[right of = D] at (60:1) (c1) {$c_1$};
        \node[right of = D] at (30:1) (c2) {$c_2$};
        \node[right of = D] at (-60:1) (c3) {$c_{n_i}$};
        \draw[->] (D) -- (c1);
        \draw[->] (D) -- (c2);
        \draw[->] (D) -- (c3);
        \foreach \angle in {5,-15,-35} {
            \node[right of = D] at (\angle:0.5) {$\cdot$};
        }
        \draw[->] (e) -- (D);
    \end{tikzpicture}
    \end{center}
where $c_j$ is the $j$th crossing encountered by traversing $\alpha_i$ along its orientation beginning at the base point. Similarly, for each $\beta_i$, assign the tensor

\begin{center}
    \begin{tikzpicture}
        \node at (0,0) (mu) {$\mu$};
        \node at (1,0) (M) {$M$};
        \node[right of = M] at (-60:1) (c1) {$c_1$};
        \node[right of = M] at (-30:1) (c2) {$c_2$};
        \node[right of = M] at (60:1) (c3) {$c_{n_i}$};
        \draw[<-] (M) -- (c1);
        \draw[<-] (M) -- (c2);
        \draw[<-] (M) -- (c3);
        \foreach \angle in {15,-5,35} {
            \node[right of = M] at (\angle:0.5) {$\cdot$};
        }
        \draw[<-] (mu) -- (M);
    \end{tikzpicture}
\end{center}
where $c_j$ is the $j$th crossing encountered by traversing $\beta_i$ along its orientation beginning at the base point.

For each link component, $L_i$, assign the tensor

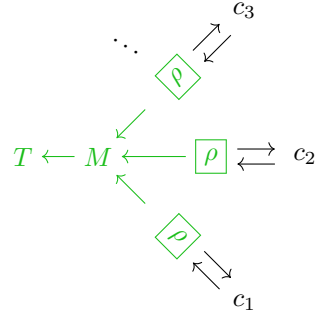
\begin{center}
    \begin{tikzpicture}
        \node[darkgreen] at (0,0) (T) {$T$};
        \node[darkgreen] at (1,0) (M) {$M$};
        \node[darkgreen, right of = M] at (45:1.5) (rho1) {\rotatebox{45}{\fbox{$\rho$}}};
        \node[darkgreen, right of = M] at (0:1.5) (rho2) {\rotatebox{0}{\fbox{$\rho$}}};
        \node[darkgreen, right of = M] at (-45:1.5) (rho3) {\rotatebox{-45}{\fbox{$\rho$}}};
        \node[right of =M] at (45:2.75) (c3) {$c_3$};
        \node[right of =M] at (0:2.75) (c2) {$c_2$};
        \node[right of =M] at (-45:2.75) (c1) {$c_1$};
        \draw[darkgreen,->] (M) -- (T);
        \draw[darkgreen,->] (rho1) -- (M);
        \draw[darkgreen,->] (rho2) -- (M);
        \draw[darkgreen,->] (rho3) -- (M);
        \node[right of =M] at (75:1.5) {\rotatebox{-30}{$\cdots$}};
        \draw[<-] ([shift={(0.35,0.2)}]rho1.center) -- ++(45:0.5);
        \draw[->] ([shift={(0.2,0.35)}]rho1.center) -- ++(45:0.5);
        \draw[<-] ([shift={(0.35,-0.1)}]rho2.center) -- ++(0:0.5);
        \draw[->] ([shift={(0.35,0.1)}]rho2.center) -- ++(0:0.5);
        \draw[->] ([shift={(0.35,-0.2)}]rho3.center) -- ++(-45:0.5);
        \draw[<-] ([shift={(0.2,-0.35)}]rho3.center) -- ++(-45:0.5);
    \end{tikzpicture}
\end{center}

where $c_j$ is the $j$th crossing encountered by traversing $L_i$ along its orientation beginning at the base point.

Before we proceed, we fix a prescription for a positive crossing with respect to each pair of curves. This is (mostly) arbitrary, but if we make a different choice, we may need to consider representations of $D(H^{\mathrm{op}})$ instead. We will use the prescription that for any pair of $\alpha_i$, $\beta_j$, $L_k$ with intersection $p$, $(d_p(\alpha_i), d_p(\beta_j))$, $(d_p(L_k), d_p(\beta_j))$, and $(d_p(\alpha_i),d_p(L_k))$ each form an ordered basis for $T_p\Sigma$, where $d_p\gamma$ is the tangent vector at a point $p$ in the direction of $\gamma$. If this basis matches the orientation of $\Sigma$ we call the crossing positive, otherwise, we call the crossing negative. Pictorially, the following crossings are positive when the plane is given the standard orientation:

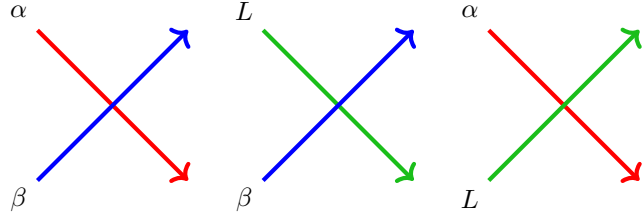
\begin{center}
    \begin{tikzpicture}
        \draw[red,ultra thick, ->] (-1,1) -- (1,-1);
        \draw[blue,ultra thick, ->] (-1,-1) -- (1,1);
        \node at (-1.25,-1.25) {$\beta$};
        \node at (-1.25,1.25) {$\alpha$};
        \draw[darkgreen,ultra thick, ->] (2,1) -- (4,-1);
        \draw[blue,ultra thick, ->] (2,-1) -- (4,1);
        \node at (1.75,-1.25) {$\beta$};
        \node at (1.75,1.25) {$L$};
        \draw[red,ultra thick, ->] (5,1) -- (7,-1);
        \draw[darkgreen,ultra thick, ->] (5,-1) -- (7,1);
        \node at (4.75,-1.25) {$L$};
        \node at (4.75,1.25) {$\alpha$};
    \end{tikzpicture}
\end{center}

For each crossing $c_{ij}$ of $\alpha_i$ with $\beta_j$, we contract by the following tensor:

\begin{center}
    \begin{tikzpicture}
        \node at (0,0) (S) {$S^{\eta_{ij}}$};
        \draw[<-] (S) -- ++(-0.75,0);
        \draw[->] (S) -- ++(0.75,0);
    \end{tikzpicture}
\end{center}
where $\eta_{ij} = 0$ if the crossing $c_{ij}$ is positive, and $\eta_{ij} = 1$ if the crossing is negative.

For each crossing $c_{ij}$ of $\alpha_i$ with $L_j$, we contract by the following tensor:

\begin{center}
    \begin{tikzpicture}
        \node at (0,0) (S) {$S^{\eta_{ij}}$};
        \node at (0,0.5) (eps) {$\epsilon$};
        \draw[<-] (S) -- ++(-0.75,0);
        \draw[->] (S) -- ++(0.75,0);
        \draw[<-] (eps) -- ++(0.75,0);
    \end{tikzpicture}
\end{center}
where $\eta_{ij} = 0$ if the crossing $c_{ij}$ is positive, and $\eta_{ij} = 1$ if the crossing is negative. Similarly, for each crossing of $\beta_i$ with $L_j$, we contract by the following tensor:

\begin{center}
    \begin{tikzpicture}
        \node at (0,0) (S) {$S^{\eta_{ij}}$};
        \node at (0,-0.5) (i) {$i$};
        \draw[->] (S) -- ++(-0.75,0);
        \draw[<-] (S) -- ++(0.75,0);
        \draw[->] (i) -- ++(0.75,0);
    \end{tikzpicture}
\end{center}
where again $\eta_{ij} = 0$ if the crossing $c_{ij}$ is positive, and $\eta_{ij} = 1$ if the crossing is negative.

Define the \emph{extended Kuperberg bracket} $\langle E\rangle_{H,(\rho,T)}$ as the contraction of this tensor network.

For brevity, and when the context is clear, we will often use the following notation relating the representation $\rho$ of the Drinfeld double $D(H)$ with its components $\rho_H$ and $\rho_{H^*}$:

\begin{center}
\begin{tikzpicture}
    \node[darkgreen] (rho) at (0,0) {\fbox{$\rho$}};
    \draw[<-] ([shift={(-0.35,0)}]rho.center) -- ++(0:-0.5);
    \draw[->,darkgreen] (rho) -- ++(0.75,0);

    \node at (1.25,0) {$=$};
    
    \begin{scope}[shift={(3,0)}]
    \node[darkgreen] (rho2) at (0,0) {\fbox{$\rho$}};
    \node (eps) at (-1,0.25) {$\epsilon$};
    \draw[<-] ([shift={(-0.35,-0.1)}]rho2.center) -- ++(10:-0.5);
    \draw[->] ([shift={(-0.35,0.1)}]rho2.center) -- (eps);
    \draw[->,darkgreen] (rho2) -- ++(0.75,0);

    \node at (1.25,0) {$=$};
    \end{scope}

    \begin{scope}[shift = {(5.75,0)}]
    \node[darkgreen] (rho3) at (0,0) {\fbox{$\rho_H$}};
    \draw[<-] ([shift={(-0.5,0)}]rho3.center) -- ++(0:-0.5);
    \draw[->,darkgreen] (rho3) -- ++(1,0);
    \end{scope}
    
\end{tikzpicture}\vspace{0.5cm}

\begin{tikzpicture}
    \node[darkgreen] (rho) at (0,0) {\fbox{$\rho$}};
    \draw[->] ([shift={(-0.35,0)}]rho.center) -- ++(0:-0.5);
    \draw[->,darkgreen] (rho) -- ++(0.75,0);

    \node at (1.25,0) {$=$};
    
    \begin{scope}[shift={(3,0)}]
    \node[darkgreen] (rho2) at (0,0) {\fbox{$\rho$}};
    \node (i) at (-1,-0.25) {$i$};
    \draw[<-] ([shift={(-0.35,-0.1)}]rho2.center) -- (i);
    \draw[->] ([shift={(-0.35,0.1)}]rho2.center) -- ++(-10:-0.5);
    \draw[->,darkgreen] (rho2) -- ++(0.75,0);

    \node at (1.25,0) {$=$};
    \end{scope}

    \begin{scope}[shift = {(5.75,0)}]
    \node[darkgreen] (rho3) at (0,0) {\fbox{$\rho_{H^*}$}};
    \draw[->] ([shift={(-0.5,0)}]rho3.center) -- ++(0:-0.5);
    \draw[->,darkgreen] (rho3) -- ++(1,0);
    \end{scope}
    
\end{tikzpicture}
\end{center}

For a pair $(M,L)$ of a closed connected oriented 3-manifold $M$ and a framed oriented link $L$ embedded in $M$, define the \emph{extended Kuperberg invariant} $Z_{\mathrm{exK}}((M,L);H,(\rho,T))$ to be
\[ Z_{\mathrm{exK}}((M,L);H,(\rho,T)) = \langle E \rangle_{H,(\rho,T)}\] for some Heegaard-Link diagram $E$ representing $(M,L)$.

\begin{mainthm}
    The bracket $\langle \bullet\rangle_{H,\mu,(\rho,T)}$ is invariant under Heegaard-Link diagram moves, and hence $Z_{\mathrm{exK}}(\bullet;H,(\rho,T))$ is an invariant of pairs $(M,L)$ consisting of a closed connected oriented 3-manifold $M$ and an oriented framed link $L$ embedded in $M$.
\end{mainthm}

\begin{proof}
    By Theorem \ref{thm:extendedHeegaardMoves}, we need only to prove invariance under the following moves:
    \begin{enumerate}[1)]
        \item Base point change
        \item Heegaard circle orientation reversal
        \item Isotopy of curve sets
        \item Handle slides
        \item Stabilization/Destabilization
    \end{enumerate}

    Note that because $e$, $\mu$, and $T$ are cyclic, we have invariance under base point change. Note also that because $S\circ e = e$ and $\mu\circ S = \mu$, and $S$ is an anti-algebra and anti-coalgebra morphism from $H$ to itself, we get invariance under circle orientation reversal.

    For isotopies of curve sets, we need the following well-known result:

    \begin{lem}
        Let $\Gamma$ be a closed 1-manifold in a surface $\Sigma$. Let $i_0$ and $i_1$ be homotopic immersions of $\Gamma$ into $\Sigma$ such that each component of $\Gamma$ is embedded in $\Sigma$ by $i_0$ and $i_1$, and the components of $i_0(\Gamma)$ and $i_1(\Gamma)$ intersect transversely, then $i_0$ and $i_1$ are diffeomorphic after a sequence of two-point and three-point moves (see Figures \ref{fig:twopoint} and \ref{fig:threepoint}).
    \end{lem}

    \begin{proof}
        By the bigon criterion (see \cite[Lemma~3.1]{Hass85}, after a sequence of two-point moves, $i_0$ and $i_1$ may be assumed to be in minimal position. Then by \cite[Lemma~3.4]{PATERSON2002205} any two minimal immersions are ambiently isotopic after a sequence of three-point moves.
    \end{proof}

    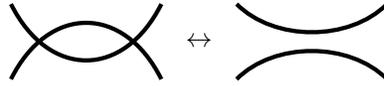
\begin{figure}
        \centering
        \begin{tikzpicture}
            \draw[ultra thick] (-1,0.5) .. controls +(0.5,-1) and +(-0.5,-1) .. (1,0.5);
            \draw[ultra thick] (-1,-0.5) .. controls +(0.5,1) and +(-0.5,1) .. (1,-0.5);
            \node[ultra thick] at (1.5,0) {$\leftrightarrow$};
            \draw[ultra thick] (2,0.5) .. controls +(0.5,-0.5) and +(-0.5,-0.5) .. (4,0.5);
            \draw[ultra thick] (2,-0.5) .. controls +(0.5,0.5) and +(-0.5,0.5) .. (4,-0.5);
        \end{tikzpicture}
        \caption{Two-Point Move}
        \label{fig:twopoint}
    \end{figure}

    \begin{figure}
        \centering
        \begin{tikzpicture}
            \draw[ultra thick] (-1.25,-0.5) -- (1.25,-0.5);
            \draw[ultra thick] (-0.5,1) -- (1,-1);
            \draw[ultra thick] (0.5,1) -- (-1,-1);
            \node at (2,0) {$\leftrightarrow$};
            \begin{scope}[shift = {(4,0)}]
            \draw[ultra thick] (-1.25,-0.5) .. controls +(1,1) and +(-1,1) .. (1.25,-0.5);
            \draw[ultra thick] (-0.5,1) .. controls +(0,-1) and +(-1,0.25) .. (1,-1);
            \draw[ultra thick] (0.5,1) .. controls +(0,-1) and +(1,0.25) .. (-1,-1);
            \end{scope}
        \end{tikzpicture}
        \caption{Three-Point Move}
        \label{fig:threepoint}
    \end{figure}
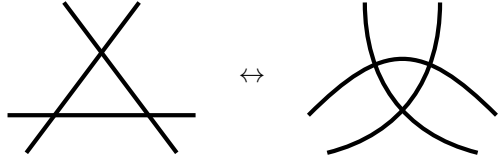

    For a two-point move with an $\alpha$ and $\beta$ curve, the following representative tensor equality demonstrates the invariance:

    \begin{center}
        \begin{tikzpicture}
            \node at (0,0) (D) {$\Delta$};
            \node at (1,0.5) (S) {$S$};
            \node at (2,0) (M) {$M$};
            \draw[->] (D) -- (S);
            \draw[->] (S) -- (M);
            \draw[->] (D) .. controls +(0.5,-0.5) and +(-0.5,-0.5) .. (M);
            \draw[<-] (D) -- ++(-0.75,0);
            \draw[->] (M) -- ++(0.75,0);
            \node at (3.5,0) {$=$};
            \begin{scope}[shift = {(5,0)}]
                \node at (0,0) (e) {$\epsilon$};
                \node at (1,0) (i) {$i$};
                \draw[<-] (e) -- ++(-0.75,0);
                \draw[->] (i) -- ++(0.75,0);
            \end{scope}
        \end{tikzpicture}
    \end{center}

    For a two-point move with an $\alpha$ curve an a link component $L_i$, the following representative tensor equality demonstrates invariance:

    \begin{center}
        \begin{tikzpicture}
            \node at (0,0) (D) {$\Delta$};
            \node at (1,0.5) (S) {$S$};
            \node[darkgreen] at (2.25,0.5) (rhoH1) {\fbox{$\rho_H$}};
            \node[darkgreen] at (2.25,-0.5) (rhoH2) {\fbox{$\rho_H$}};
            \node[darkgreen] at (3.5,0) (M) {$M$};
            \draw[->] (D) -- (S);
            \draw[->] (S) -- (rhoH1);
            \draw[->] (D) -- (rhoH2);
            \draw[->,darkgreen] (rhoH1) -- (M);
            \draw[->,darkgreen] (rhoH2) -- (M);
            \draw[<-] (D) -- ++(-0.75,0);
            \draw[->,darkgreen] (M) -- ++(0.75,0);
            \node at (5,0) {$=$};
            \begin{scope}[shift = {(6.5,0)}]
            \node at (0,0) (D) {$\Delta$};
            \node at (1,0.5) (S) {$S$};
            \node[darkgreen] at (3.25,0) (rhoH) {\fbox{$\rho_H$}};
            \node at (2,0) (M) {$M$};
            \draw[->] (D) -- (S);
            \draw[->] (S) -- (M);
            \draw[->] (D) .. controls +(0.5,-0.5) and +(-0.5,-0.5) .. (M);
            \draw[->] (M) -- (rhoH);
            \draw[<-] (D) -- ++(-0.75,0);
            \draw[->,darkgreen] (rhoH) -- ++(0.75,0);
            \end{scope}
            \node at (5,-1.5) {$=$};
            \begin{scope}[shift={(6.5,-1.5)}]
                \node at (0,0) (eps) {$\epsilon$};
                \node at (1,0) (i) {$i$};
                \node[darkgreen] at (2.25,0) (rhoH) {\fbox{$\rho_H$}};
                \draw[<-] (eps) -- ++(-0.75,0);
                \draw[->] (i) -- (rhoH);
                \draw[->,darkgreen] (rhoH) -- ++(0.75,0);
            \end{scope}
            \node at (5,-3) {$=$};
            \begin{scope}[shift={(6.5,-3)}]
                \node at (0,0) (eps) {$\epsilon$};
                \node[darkgreen] at (2.25,0) (I) {\fbox{$\mathrm{Id}_{V}$}};
                \draw[<-] (eps) -- ++(-0.75,0);
                \draw[->,darkgreen] (I) -- ++(0.75,0);
            \end{scope}
        \end{tikzpicture}
    \end{center}

    For a representative three-point move:

    
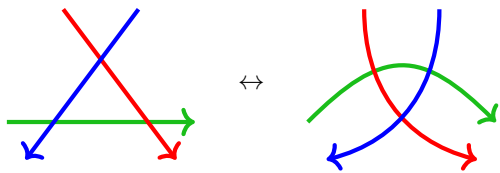
\begin{center}
        \begin{tikzpicture}
            \draw[->,darkgreen,ultra thick] (-1.25,-0.5) -- (1.25,-0.5);
            \draw[->,red,ultra thick] (-0.5,1) -- (1,-1);
            \draw[->,blue,ultra thick] (0.5,1) -- (-1,-1);
            \node at (2,0) {$\leftrightarrow$};
            \begin{scope}[shift = {(4,0)}]
            \draw[->,darkgreen,ultra thick] (-1.25,-0.5) .. controls +(1,1) and +(-1,1) .. (1.25,-0.5);
            \draw[->,red,ultra thick] (-0.5,1) .. controls +(0,-1) and +(-1,0.25) .. (1,-1);
            \draw[->,blue,ultra thick] (0.5,1) .. controls +(0,-1) and +(1,0.25) .. (-1,-1);
            \end{scope}
        \end{tikzpicture}
    \end{center}

    the following tensor equality demonstrates invariance:

    \begin{center}
        \begin{tikzpicture}
            \node at (0,0) (D) {$\Delta$};
            \node at (2,0) (M) {$M$};
            \node[darkgreen] at (2,-2) (rhoHdual) {\fbox{$\rho_{H^*}$}};
            \node[darkgreen] at (0,-2) (rhoH) {\fbox{$\rho_H$}};
            \node at (0,-1) (S) {$S$};
            \node[darkgreen] at (1,-3) (Mgreen) {$M$};
            \draw[->] (D) -- (M);
            \draw[->] (D) -- (S);
            \draw[->] (S) -- (rhoH);
            \draw[<-] (M) -- (rhoHdual);
            \draw[<-] (D) -- ++(-0.75,0);
            \draw[->] (M) -- ++(0.75,0);
            \draw[->,darkgreen] (rhoH) -- (Mgreen);
            \draw[->,darkgreen] (rhoHdual) -- (Mgreen);
            \draw[->,darkgreen] (Mgreen) -- ++(0,-0.75);
            \node at (3.5,-1) {$=$};
            \begin{scope}[shift = {(5,0)}]
            \node at (0,0) (D) {$\Delta^{\mathrm{cop}}$};
            \node at (2,0) (M) {$M^{\mathrm{op}}$};
            \node[darkgreen] at (2,-2) (rhoHdual) {\fbox{$\rho_{H^*}$}};
            \node[darkgreen] at (0,-2) (rhoH) {\fbox{$\rho_H$}};
            \node at (0,-1) (S) {$S$};
            \node[darkgreen] at (1,-3) (Mgreen) {$M^{\mathrm{op}}$};
            \draw[->] (D) -- (M);
            \draw[->] (D) -- (S);
            \draw[->] (S) -- (rhoH);
            \draw[<-] (M) -- (rhoHdual);
            \draw[<-] (D) -- ++(-0.75,0);
            \draw[->] (M) -- ++(0.75,0);
            \draw[->,darkgreen] (rhoH) -- (Mgreen);
            \draw[->,darkgreen] (rhoHdual) -- (Mgreen);
            \draw[->,darkgreen] (Mgreen) -- ++(0,-0.75);
            \end{scope}
        \end{tikzpicture}
    \end{center}

    which is true by Lemma \ref{lem:doubleequalities}. One can similarly use Lemma \ref{lem:doubleequalities} to show invariance under the other versions of three-point moves.

    For a handle slide, suppose we first slide $\beta_i$ over $\beta_j$:

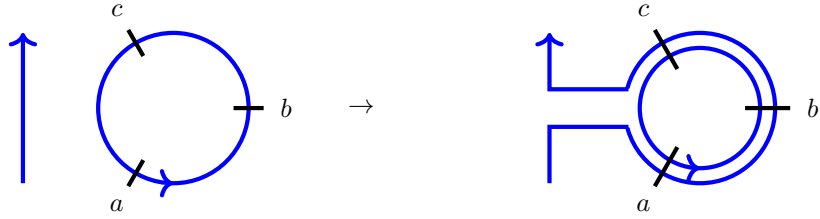
    \begin{center}
    \begin{tikzpicture}
        \draw[blue,ultra thick,<-] (0,1) -- (0,-1);
        \begin{scope}[shift = {(2,0)}]
            \draw[blue,ultra thick,decoration={markings, mark=at position 0.75 with {\arrow{>}}}, postaction={decorate}] (0,0) circle (1cm);
            \foreach \angle in {0,-120,120} {
            \draw[ultra thick] (\angle : 0.8) -- (\angle : 1.2);
        }
        \node at (0 : 1.5) {$b$};
        \node at (120 : 1.5) {$c$};
        \node at (-1200 : 1.5) {$a$};
        \end{scope}
        \node at (4.5,0) {$\rightarrow$};
        \begin{scope}[shift = {(7,0)}]
        \begin{scope}[shift = {(2,0)}]
        \draw[blue,ultra thick,decoration={markings, mark=at position 0.75 with {\arrow{>}}}, postaction={decorate}] (0,0) circle (0.8cm);
        \draw[blue,ultra thick] (0,0) circle (1cm);
        \foreach \angle in {0,-120,120} {
            \draw[ultra thick] (\angle : 0.6) -- (\angle : 1.2);
        }
        \node at (0 : 1.5) {$b$};
        \node at (120 : 1.5) {$c$};
        \node at (-1200 : 1.5) {$a$};
        \draw[fill=white,draw=white] (-1.2,-0.25) rectangle (-0.9,0.25);
        \draw[blue,ultra thick] (-2,-1) -- (-2,-0.25) -- (-0.95,-0.25);
        \draw[blue,ultra thick,<-] (-2,1) -- (-2,0.25) -- (-0.95,0.25);
        \end{scope}
        \end{scope}
    \end{tikzpicture}
    \end{center}

    The ``after'' picture corresponds to the following tensor:

    \begin{center}
        \begin{tikzpicture}
            \node at (0,0.5) (a) {$a$};
            \node at (0,0) (b) {$b$};
            \node at (0,-0.5) (c) {$c$};
            \node at (1,0.5) (Da) {$\Delta$};
            \node at (1,0) (Db) {$\Delta$};
            \node at (1,-0.5) (Dc) {$\Delta$};
            \node at (2.5,0.5) (MU) {$M$};
            \node at (2.5,-0.5) (ML) {$M$};
            \node at (3.5,-0.5) (mu) {$\mu$};
            \draw[->] (a) -- (Da);
            \draw[->] (b) -- (Db);
            \draw[->] (c) -- (Dc);
            \draw[->] (Da) -- (MU);
            \draw[->] (Da) -- (ML);
            \draw[->] (Db) -- (MU);
            \draw[->] (Db) -- (ML);
            \draw[->] (Dc) -- (MU);
            \draw[->] (Dc) -- (ML);
            \draw[->] (MU) -- ++(0.75,0);
            \draw[->] (ML) -- (mu);
        \end{tikzpicture}
    \end{center}

    The following equalities demonstrate invariance:

    \begin{center}
        \begin{tikzpicture}
            \node at (0,0.5) (a) {$a$};
            \node at (0,0) (b) {$b$};
            \node at (0,-0.5) (c) {$c$};
            \node at (1,0.5) (Da) {$\Delta$};
            \node at (1,0) (Db) {$\Delta$};
            \node at (1,-0.5) (Dc) {$\Delta$};
            \node at (2.5,0.5) (MU) {$M$};
            \node at (2.5,-0.5) (ML) {$M$};
            \node at (3.5,-0.5) (mu) {$\mu$};
            \draw[->] (a) -- (Da);
            \draw[->] (b) -- (Db);
            \draw[->] (c) -- (Dc);
            \draw[->] (Da) -- (MU);
            \draw[->] (Da) -- (ML);
            \draw[->] (Db) -- (MU);
            \draw[->] (Db) -- (ML);
            \draw[->] (Dc) -- (MU);
            \draw[->] (Dc) -- (ML);
            \draw[->] (MU) -- ++(0.75,0);
            \draw[->] (ML) -- (mu);
            \node at (4.5,0) {$=$};
            \begin{scope}[shift = {(5.5,0)}]
            \node at (0,0.5) (a) {$a$};
            \node at (0,0) (b) {$b$};
            \node at (0,-0.5) (c) {$c$};
            \node at (2,0) (D) {$\Delta$};
            \node at (1,0) (M) {$M$};
            \node[below right =0.25 of D] (mu) {$\mu$};
            \draw[->] (a) -- (M);
            \draw[->] (b) -- (M);
            \draw[->] (c) -- (M);
            \draw[->] (M) -- (D);
            \draw[->] (D) -- (mu);
            \draw[->] (D) -- ++(45:0.75);
            \end{scope}
            \node at (4.5,-2) {$=$};
            \begin{scope}[shift = {(5.5,-2)}]
            \node at (0,0.5) (a) {$a$};
            \node at (0,0) (b) {$b$};
            \node at (0,-0.5) (c) {$c$};
            \node at (2,0) (mu) {$\mu$};
            \node at (1,0) (M) {$M$};
            \node at (3,0) (i) {$i$};
            \draw[->] (a) -- (M);
            \draw[->] (b) -- (M);
            \draw[->] (c) -- (M);
            \draw[->] (M) -- (mu);
            \draw[->] (i) -- ++(0.75,0);
            \end{scope}
        \end{tikzpicture}
    \end{center}

    Note that if instead $L_k$ slides over a Heegaard circle we can use the fact that $\rho$ is a representation to reduce the computation to the above, as we did in proving invariance under the two-point move.

    Finally, we check stabilization. Note that destabilization is the inverse operation, so we need only check stabilization. This follows because $\mu(e) = 1$.
\end{proof}

\begin{cor}\label{cor:recoverKup}
    The extended Kuperberg invariant recovers the Kuperberg invariant (see \cite{Kup91}) when we take $L$ to be the empty link.
\end{cor}

\section{Relationship with the Hennings Invariant}
\label{Hennings_Sec}
Hennings \cite{hen96}, and later reformulated by Kauffman and Radford \cite{kauffman1995invariants}, defined a 3-manifold invariant from unimodular ribbon Hopf algebras where the 3-manifold is presented as a surgery link. Let $H$ be a ribbon Hopf algebra and suppose that $H$ is involutory, and thus unimodular. Let $\mu\in H^*$ be a non-zero two-sided integral. We associate a regular isotopy invariant $\langle L\rangle _{H,\mu}$ to a framed unoriented link $L$ as follows. Choose a link diagram for $L$ with respect to a height function such that the crossings are not critical points. On each component $L_i$ of $L$, pick a base point which is neither a crossing nor an extremum, and arbitrarily orient $L_i$. Define $\delta_i$ to be 0 if the orientation of $L_i$ near the base point is downwards and 1 otherwise. For a point $p$ on $L_i$ which is not an extremum, let $w_p$ be the algebraic sum of extrema between the base point and $p$, where an extremum is counted as +1 (resp. -1) if the orientation near it is counterclockwise (resp. clockwise). Equivalently, $w_p$ is 2 times the total counterclockwise rotation, in units of $1 = 360^\circ$, of the tangent of $L_i$ from the base point to $p$. Define $w_i$ to be $w_p/2$ for $p$ very close to the base point in the backward direction of $L_i$. Clearly $w_i$ is the winding number of $L_i$. Decorate each crossing with the tensor factors of the $R$-matrix for $H$ ($R = \sum_i s_i\otimes t_i$) as below.

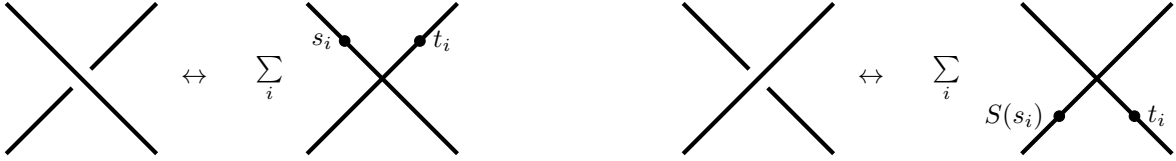
\begin{center}
    \begin{tikzpicture}
        \draw[ultra thick] (0,-1) -- (2,1);
        \draw[line width = 10pt, white] (0,1) -- (2,-1);
        \draw[ultra thick] (0,1) -- (2,-1);
        \node at (2.5,0) {$\leftrightarrow$};
        \node at (3.5,0) {$\underset{i}{\sum}$};
        \draw[ultra thick] (4,-1) -- (6,1);
        \draw[ultra thick] (6,-1) -- (4,1);
        \draw[fill = black] (4.5,0.5) circle (2pt);
        \draw[fill = black] (5.5,0.5) circle (2pt);
        \node at (4.2,0.5) {$s_i$};
        \node at (5.8,0.5) {$t_i$};
        \begin{scope}[shift = {(9,0)}]
        \draw[ultra thick] (0,1) -- (2,-1);
        \draw[white, line width = 10pt] (0,-1) -- (2,1);
        \draw[ultra thick] (0,-1) -- (2,1);
        \node at (2.5,0) {$\leftrightarrow$};
        \node at (3.5,0) {$\underset{i}{\sum}$};
        \draw[ultra thick] (4.5,-1) -- (6.5,1);
        \draw[ultra thick] (6.5,-1) -- (4.5,1);
        \draw[fill = black] (5,-0.5) circle (2pt);
        \draw[fill = black] (6,-0.5) circle (2pt);
        \node at (4.4,-0.5) {$S(s_i)$};
        \node at (6.3,-0.5) {$t_i$};
        \end{scope}
    \end{tikzpicture}
\end{center}

Then we replace each decorating element $x$ on $L_i$ by $S^{-w_{p(x)}+\delta_i}(x)$, where $p(x)$ denotes the point on $L_i$ where $x$ is located.

Then the Hennings-Kauffman-Radford Link Invariant $\langle L\rangle _{H,\mu}$ is the evaluation of the right integral $\mu_R$ on the products along each $L_i$: $$\langle L\rangle _{H,\mu} := \sum_{(R)} \mu(q_1)\cdots \mu(q_{c(L)})$$ where $c(L)$ is the number of components of $L$, and $q_i\in H$ is the product of the decorating elements (after applying $S$ powers) on $L_i$ multiplied in the order following its orientation starting from the base point.

We may upgrade this link invariant to one of the 3-manifold obtained by surgery on $S^3$ along the framed link $L$ by introducing the following normalization factor: Let $v\in H$ be the ribbon element, and let $\omega(v)$ be a square root of $\mu(v)/\mu(v^{-1})$, then  the \emph{Hennings-Kauffman-Radford invariant} for the 3-manifold $M$ obtained by surgery on $S^3$ along the framed link $L$ is defined to be: \[Z_{\mathrm{HKR}}(M;H,\mu,\omega(v)) = (\mu(v)/\omega(v))^{-c(L)}\omega(v)^{-\text{sign}(L)}\langle L\rangle_{H,\mu}\] where $\text{sign}(L)$ is the signature of the framing matrix of $L$.

\begin{mainthm}\label{hennings_cor}
Let $H$ be an involutory Hopf algebra with nonzero two-sided integral $\mu$ and nonzero two-sided cointegral $e$ such that $\mu(e) =1$. Then for an (unoriented) framed link $L$ embedded in $S^3$, \[Z_{\mathrm{exK}}((S^3,L);H, (\rho_{\mathrm{reg}},\widetilde{T})) = Z_{\mathrm{HKR}}(S^3(L);D(H),e\otimes \mu,1)\]
\end{mainthm}

\begin{proof}
   As above, choose a diagram for $L$ in braid closure such that the blackboard framing agrees with the endowed framing. On each component of $L$, pick a base point which is neither a crossing nor an extremum, and orient $L$ so that all strands of the braid are directed downward.
   
   We now embed $L$ on a Heegaard surface for $S^3$ in the following way. Embed the link in $S^2$ and to each crossing, stabilize as in \ref{fig:resolvingcrossings}. The result is a Heegaard-Link diagram for the pair $(S^3,L)$ where the blackboard framing of $L$ matches the endowed framing. See Figure \ref{fig:resolvingcrossings}.

    \begin{figure}
        \centering
        \begin{tikzpicture}
            \draw[->,darkgreen,ultra thick] (-1,1) -- (1,-1);
            \draw[darkgreen,ultra thick] (1,1) -- (0.5,0.5);
            \draw[darkgreen,->,ultra thick] (-0.5,-0.5) -- (-1,-1);
            \draw[blue,ultra thick] (0.5,0.5) -- (-0.5,-0.5);
            \draw[draw=black, fill=lightgray] (0.5,0.5) circle (0.25);
            \draw[draw=black, fill=lightgray] (-0.5,-0.5) circle (0.25);
            \draw[draw=red,ultra thick] (-0.5,-0.5) circle (0.25);
        \end{tikzpicture}
        \caption{Stabilizing bridge at a crossing}
        \label{fig:resolvingcrossings}
    \end{figure}
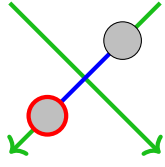

    Let $\rho = \rho_{\mathrm{reg}}$,

    \begin{center}
        \begin{tikzpicture}
            \node at (0,0) (D) {$\Delta$};
            \node at (1,0) (M) {$M$};
            \node[darkgreen] at (2,0) (rhoR) {\fbox{$\rho$}};
            \node[darkgreen] at (-1,0) (rhoL) {\fbox{$\rho$}};
            \node at (0,-1) (e) {$e$};
            \node at (1,-1) (mu) {$\mu$};
            \draw[->] (D) -- (M);
            \draw[->] (D) -- (rhoL);
            \draw[<-] (M) -- (rhoR);
            \draw[->] (e) -- (D);
            \draw[->] (M) -- (mu);
            \draw[darkgreen,->] (rhoR) -- ++(0.75,0);
            \draw[darkgreen,->] (rhoL) -- ++(-0.75,0);
            \node at (3.4,0) {$=$};
            \begin{scope}[shift = {(5,0)}]
                \node at (0,0) (S) {$S$};
                \draw[->] (S) -- ++(-0.75,0);
                \draw[<-] (S) -- ++(0.75,0);
                \node at (2,0) {$\mu(e)$};
            \end{scope}
        \end{tikzpicture}
    \end{center}

    For a positive crossing (as in Figure \ref{fig:resolvingcrossings}), the contribution to the invariant is thus an introduction of the inverse $R$-matrix $R^D$ for $D(H)$, where the first component of $R^D$ corresponds to the overcrossing. Similarly, if the crossing is negative, we introduce the $R$-matrix for $D(H)$. The introduction of the $S^{-w_{p(x)}+\delta_i}$ tensors in the HKR invariant are canceled since their only contributions occur in the cups and caps of the diagram and each cup is paired with a cap next to one another in the trace closure. 

    Now, since $\rho$ is the regular representation, it is given by $x\mapsto L_x$, or left multiplication in the Drinfeld double, by Lemma \ref{lem:traceisintegral}, we have that contracting the tensors is given by the integral $e\otimes \mu$ applied to the multiplication along the corresponding link component in the direction of orientation starting from the base point.

    All that is left to check is the normalization factor. Recall that by Proposition \ref{lem:uequalsv}, the ribbon element $v$ may be taken to be the Drinfeld element $u$. Taking $\mu^D(v)$, we obtain

    \begin{center}
        \begin{tikzpicture}
        \node at (0,0) (R) {$R$};
        \node at (0.5,0.4) {\tiny{$1$}};
        \node at (1.25,0.5) (S) {$S$};
        \node at (2,0) (M) {$M$};
        \node at (3,0) (mu) {$\mu^D$};
        \draw[->] (R) .. controls +(0.5,0.5) and +(-0.5,-0.5) .. (M);
        \draw[->] (R) .. controls +(0.5,-0.5) and +(-0.5,-0.5) .. (S);
        \draw[->] (S) -- (M);
        \draw[->] (M) -- (mu);
        \begin{scope}[shift = {(2,-2)}]
            \node at (-1,0) {$=$};
            \node at (0,1) (eps) {$\epsilon$};
            \node at (0,-1) (i) {$i$};
            \draw[] (0.75,-0.75) arc (-90:-270:0.75);
        \begin{scope}[shift = {(2,-1)}]
            \node at (0,0.25) (SUL) {$S$};
            \node at (0,-0.25) (SBL) {$S$};
            \node at (1,-0.25) (D) {$\Delta$};
            \node at (2.5,0.25) (M) {$M$};
            \node at (1.75,-0.75) (SB) {$S$};
            \draw[] (SUL) -- ++(-1.25,0);
            \draw[<-] (SBL) -- (i);
            \draw[->] (M) -- (SUL);
            \draw[->] (SBL) -- (D);
            \draw[->] (D) -- ++(0,-0.5) -- (SB);
            \draw[->] (SB) -- ++(0.75,0) -- (M);
            \draw[->] (D) -- ++(0,1) -- ++(1.5,0) -- (M);
            \draw[<-] (M) -- ++(0.75,0);
            \draw[] (D) -- ++(2.25,0) -- ++(1,1.75);
        \end{scope}
        \begin{scope}[shift = {(7,0)}]
                \node at (0,0.5) (DL) {$\Delta$};
                \node at (1,-0.5) (S) {$S$};
                \node at (2,0) (ML) {$M$};
                \node at (3.5,-0.25) (MR) {$M$};
                \node at (3.5,0.25) (DR) {$\Delta^{\mathrm{cop}}$};
                \draw[<-] (DL) -- ++(-0.75,0);
                \draw[->] (DL) -- ++(0,-1) -- (S);
                \draw[->] (DL) -- ++(2,0) -- (ML);
                \draw[->] (S) -- ++(1,0) -- (ML);
                \draw[->] (ML) -- ++(-2.25,0) -- ++(-2,1) -- (eps);
                \draw[<-] (MR) -- ++(0,-0.5) -- ++(-4.25,0) -- ++(-1.5,1) -- ++(-4,0.5);
                \draw[] (DR) -- ++(0,0.5) -- ++(-4.5,0) -- ++(-0.75,-1.5);
                \draw[->] (DR) -- (ML);
                \draw[->] (DL) -- ++(2.25,-1) -- ++(0.5,0) -- (MR);
                \node at (4.5,0.25) (e) {$e$};
                \node at (4.5,-0.25) (mu) {$\mu$};
                \draw[<-] (DR) -- (e);
                \draw[->] (MR) -- (mu);
            \end{scope}
        \end{scope}

        \begin{scope}[shift = {(2,-5)}]
            \node at (-1,0) {$=$};
            \draw[] (0.75,-0.75) arc (-90:-270:0.75);
        \begin{scope}[shift = {(2,-1)}]
            \node at (0,0.25) (SUL) {$S$};
            \node at (2.5,0.25) (M) {$M$};
            \draw[] (SUL) -- ++(-1.25,0);
            \draw[->] (M) -- (SUL);
            \draw[<-] (M) -- ++(0.75,0);
        \end{scope}
        \begin{scope}[shift = {(7,0)}]
                \node at (3.5,-0.25) (MR) {$M$};
                \node at (3.5,0.25) (DR) {$\Delta^{\mathrm{cop}}$};
                \draw[<-] (DL) -- ++(-0.75,0);
                \draw[<-] (MR) -- ++(0,-0.5) -- ++(-4.25,0) -- ++(-1.5,1) -- ++(-4,0.5);
                \draw[] (DR) -- ++(0,0.5) -- ++(-4.5,0) -- ++(-0.75,-1.5);
                \node at (4.5,0.25) (e) {$e$};
                \node at (4.5,-0.25) (mu) {$\mu$};
                \draw[<-] (DR) -- (e);
                \draw[->] (MR) -- (mu);
            \end{scope}
        \end{scope}
    \end{tikzpicture}
    \end{center}

    This is just $\mu(e) = 1$. One can similarly show that $\mu^D(v^{-1}) = 1$, which, together with the computations at crossings, yields the desired result.
\end{proof}

\begin{mainthm}
    Let $H$ be an involutory ribbon Hopf algebra with universal R-matrix $R$ and nonzero two-sided integral $\mu$ and nonzero two-sided cointegral $e$ such that $\mu(e) =1$. Define a representation $\rho_R:D(H)\to \Endo(H)$ of the Drinfeld double of $H$ by

    \begin{center}
        \begin{tikzpicture}
            \node[darkgreen] at (0,0) (rho) {\fbox{$\rho_R$}};
            \draw[<-] ([shift={(-0.4,-0.1)}]rho.center) -- ++(-0.5,0);
            \draw[->] ([shift={(-0.4,0.1)}]rho.center) -- ++(-0.5,0);
            \draw[->,darkgreen] (rho) -- ++(0.75,0);
            \node at (1.5,0) {$:=$};
            \begin{scope}[shift = {(3,0)}]
                \node at (0,0) (R) {$R$};
                \node at (0.2,0.4) {\tiny{$1$}};
                \node at (1,-0.5) (M) {$M$};
                \draw[->] (R) .. controls +(0.5,0.5) and +(1,0) .. (-0.5,0.75); 
                \draw[->] (R) -- (M);
                \draw[<-] (M) -- ++(-170:1.5);
                \draw[->] (M) -- ++(0.75,0);
            \end{scope}
        \end{tikzpicture}
    \end{center}
    
    then
    \[Z_{\mathrm{exK}}((S^3,L);H,(\rho_R,\widetilde{T})) = Z_{\mathrm{HKR}}(S^3(L);H,\mu,1)\]
\end{mainthm}

\begin{proof}
    First we note that $\rho_R$ indeed forms a representation of the Drinfeld double (c.f. \cite{Rad11}).
    Using the same Heegaard-Link diagram as in the proof of Theorem \ref{hennings_cor}, we obtain the following tensor at positive link crossings. 

    \begin{center}
    \begin{tikzpicture}
            \node at (0,0) (D) {$\Delta$};
            \node at (1,0) (M) {$M$};
            \node[darkgreen] at (2,0) (rhoR) {\fbox{$\rho$}};
            \node[darkgreen] at (-1,0) (rhoL) {\fbox{$\rho$}};
            \node at (0,-1) (e) {$e$};
            \node at (1,-1) (mu) {$\mu$};
            \draw[->] (D) -- (M);
            \draw[->] (D) -- (rhoL);
            \draw[<-] (M) -- (rhoR);
            \draw[->] (e) -- (D);
            \draw[->] (M) -- (mu);
            \draw[darkgreen,->] (rhoR) -- ++(0.75,0);
            \draw[darkgreen,->] (rhoL) -- ++(-0.75,0);
    \end{tikzpicture}
    \end{center}
    
    By Lemma \ref{lem:integralequalities}, using $\rho = \rho_R$, and plugging in we have
    \begin{center}
        \begin{tikzpicture}
            \node at (0,0) (D) {$\Delta$};
            \node at (1,0) (M) {$M$};
            \node[darkgreen] at (2,0) (rhoR) {\fbox{$\rho$}};
            \node[darkgreen] at (-1,0) (rhoL) {\fbox{$\rho$}};
            \node at (0,-1) (e) {$e$};
            \node at (1,-1) (mu) {$\mu$};
            \draw[->] (D) -- (M);
            \draw[->] (D) -- (rhoL);
            \draw[<-] (M) -- (rhoR);
            \draw[->] (e) -- (D);
            \draw[->] (M) -- (mu);
            \draw[darkgreen,->] (rhoR) -- ++(0.75,0);
            \draw[darkgreen,->] (rhoL) -- ++(-0.75,0);
            \node at (3.5,0) {$=$};
            \begin{scope}[shift = {(6,0)}]
                \node[darkgreen] at (1,0) (rhoR) {\fbox{$\rho_R$}};
                \node[darkgreen] at (-1,0) (rhoL) {\fbox{$\rho_R$}};
                \node at (0,0) (S) {$S$};
                \draw[->] (rhoR) -- (S);
                \draw[->] (S) -- (rhoL);
                \draw[->,darkgreen] (rhoL) -- ++(-0.75,0);
                \draw[->,darkgreen] (rhoR) -- ++(0.75,0);
                \node at (3,0) {$\mu(e)$};
            \end{scope}
            \begin{scope}[shift = {(8,-2.5)}]
                \node at (-4.5,0) {$=$};
                \node at (0,0) (R) {$R$};
                \node at (-0.3,0.2) {\tiny{$1$}};
                \node at (1,-0.5) (M) {$M$};
                \node at (0,-0.5) (i) {$i$};
                \node at (-1,0) (S) {$S$};
                \draw[->] (R) -- (S); 
                \draw[->] (R) -- (M);
                \draw[<-] (M) -- (i);
                \draw[->] (M) -- ++(0.75,0);
            \begin{scope}[shift = {(-2,0.5)}]
                \node at (0,0) (R) {$R$};
                \node at (0,0.4) {\tiny{$1$}};
                \node at (-1,-0.5) (M) {$M^{\mathrm{op}}$};
                \node at (1,0.75) (eps) {$\epsilon$};
                \draw[->] (R) .. controls +(-0.5,0.5) and +(-1,0) .. (eps); 
                \draw[->] (R) -- (M);
                \draw[<-] (M) -- (S);
                \draw[->] (M) -- ++(-0.75,0);
            \end{scope}
            \end{scope}
            \begin{scope}[shift = {(6,-4)}]
                \node at (-2.5,0) {$=$};
                \node at (0,0) (R) {$R$};
                \node at (-0.3,0.2) {\tiny{$1$}};
                \node at (-1,0) (S) {$S$};
                \draw[->] (R) -- (S);
                \draw[->] (S) -- ++(-0.75,0);
                \draw[->] (R) -- ++(0.75,0);
            \end{scope}
        \end{tikzpicture}
    \end{center}
     
    Thus the contribution to the invariant is the inverse R-matrix with the first leg corresponding to the overcrossing as before. For the negative crossings, we obtain a contribution of the R-matrix, and thus, as in \ref{hennings_cor}, we obtain the HKR invariant of $L$ with respect to $H$.
\end{proof}

\section{Relationship with Surgery}
\label{Surgery_Sec}
Let $(M,L)$ be a pair of a closed connected oriented 3-manifold $M$ and an oriented framed link $L$ embedded in $M$. Recall that using Theorem \ref{thm:surgerydiagram} we obtain a Heegaard diagram for the surgery manifold $M(L)$. 

\begin{mainthm}\label{main_theorem}
    For a pair $(M,L)$,
    \[Z_{\mathrm{exK}}((M,L);H,(\rho_{\mathrm{reg}},\widetilde{T})) = Z_{\mathrm{exK}}((M(L),\emptyset);H,\emptyset)\]
\end{mainthm}

\begin{proof}
    Without loss of generality, assume that $L$ is a knot. Then recall Figure \ref{fig:SurgeryHeegaardDiagram}:
    
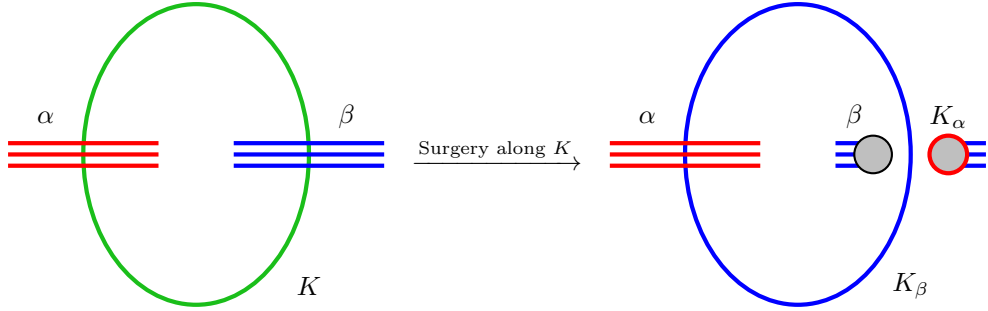
    \begin{center}
    \begin{tikzpicture}
        \draw[darkgreen, ultra thick] (0,0) ellipse (1.5cm and 2cm);
        \draw[red, ultra thick] (-0.5,0) -- (-2.5,0);
        \draw[red, ultra thick] (-0.5,0.15) -- (-2.5,0.15);
        \draw[red, ultra thick] (-0.5,-0.15) -- (-2.5,-0.15);
        \draw[blue, ultra thick] (0.5,0) -- (2.5,0);
        \draw[blue, ultra thick] (0.5,0.15) -- (2.5,0.15);
        \draw[blue, ultra thick] (0.5,-0.15) -- (2.5,-0.15);
        \node at (1.5,-1.75) {$K$};
        \node at (-2,0.5) {$\alpha$};
        \node at (2,0.5) {$\beta$};
        \node at (4,0) {$\xrightarrow{\text{Surgery along }K}$};
    \begin{scope}[shift = {(8,0)}]
        \draw[blue, ultra thick] (0,0) ellipse (1.5cm and 2cm);
        \draw[red, ultra thick] (-0.5,0) -- (-2.5,0);
        \draw[red, ultra thick] (-0.5,0.15) -- (-2.5,0.15);
        \draw[red, ultra thick] (-0.5,-0.15) -- (-2.5,-0.15);
        \draw[blue, ultra thick] (0.5,0) -- (1,0);
        \draw[blue, ultra thick] (0.5,0.15) -- (1,0.15);
        \draw[blue, ultra thick] (0.5,-0.15) -- (1,-0.15);
        \draw[blue, ultra thick] (2,0) -- (2.5,0);
        \draw[blue, ultra thick] (2,0.15) -- (2.5,0.15);
        \draw[blue, ultra thick] (2,-0.15) -- (2.5,-0.15);
        \draw[draw=black, fill=lightgray, thick] (1,0) circle (0.25);
        \draw[draw=black, fill=lightgray, thick] (2,0) circle (0.25);
        \draw[red, ultra thick] (2,0) circle (0.25);
        \node at (1.5,-1.75) {$K_{\beta}$};
        \node at (-2,0.5) {$\alpha$};
        \node at (0.75,0.5) {$\beta$};
        \node at (2,0.5) {$K_\alpha$};
        \end{scope}
    \end{tikzpicture}
    \end{center}

    The left side corresponds to the tensor

    \begin{center}
        \begin{tikzpicture}
            \node at (-1,0) (M) {$M$};
            \node at (4,0) (D) {$\Delta$};
            \node[darkgreen] at (0.25,0) (rhoH) {\fbox{$\rho_H$}};
            \node[darkgreen] at (2.75,0) (rhoHdual) {\fbox{$\rho_{H^*}$}};
            \node[darkgreen] at (1.5,0) (Mgreen) {$M$};
            \node[darkgreen] at (1.5,1) (int) {$\widetilde{T}$};
            \foreach \angle in {135,180,-135} {
                \draw[<-] (M) -- ++(\angle:0.75);
            }
            \foreach \angle in {45,0,-45} {
                \draw[->] (D) -- ++(\angle:0.75);
            }
            \draw[->] (M) -- (rhoH);
            \draw[->] (rhoHdual) -- (D);
            \draw[->,darkgreen] (rhoH) -- (Mgreen);
            \draw[->,darkgreen] (rhoHdual) -- (Mgreen);
            \draw[->,darkgreen] (Mgreen) -- (int);
        \end{tikzpicture}   
    \end{center}

    By Lemma \ref{lem:traceisintegral} and that $\rho$ is a representation of the Drinfeld double, we have that plugging in the relevant structures and simplifying yields:

    \begin{center}
        \begin{tikzpicture}
            \node at (0,0.75) (M) {$M$};
            \node at (0,-0.75) (D) {$\Delta$};
            \node at (4,0) (e) {$e$};
            \node at (4,-0.5) (mu) {$\mu$};
            \foreach \angle in {135,180,-135} {
                \draw[<-] (M) -- ++(\angle:0.75);
            }
            \foreach \angle in {135,180,-135} {
                \draw[->] (D) -- ++(\angle:0.75);
            }
            \begin{scope}[shift = {(1,0)}]
                \node at (0,0.5) (DL) {$\Delta$};
                \node at (1,-0.5) (S) {$S$};
                \node at (2,0) (ML) {$M$};
                \draw[<-] (DL) -- (M);
                \draw[->] (DL) -- ++(0,-1) -- (S);
                \draw[->] (DL) -- ++(2,0) -- (ML);
                \draw[->] (S) -- ++(1,0) -- (ML);
                \draw[->] (ML) -- ++(-2.25,0) -- ++(0,-0.5) -- (D);
                \draw[->] (DL) -- ++(2.25,-1) -- (mu);
                \draw[<-] (ML) -- (e);
            \end{scope}
            \node at (5,0) {$=$};
            \begin{scope}[shift = {(6.5,0)}]
                \node at (0,0.75) (M) {$M$};
            \node at (0,-0.75) (D) {$\Delta$};
            \node at (1,-0.75) (e) {$e$};
            \node at (1,0.75) (mu) {$\mu$};
            \foreach \angle in {135,180,-135} {
                \draw[<-] (M) -- ++(\angle:0.75);
            }
            \foreach \angle in {135,180,-135} {
                \draw[->] (D) -- ++(\angle:0.75);
            }
            \draw[->] (e) -- (D);
            \draw[->] (M) -- (mu);
            \end{scope}
        \end{tikzpicture}
    \end{center}

    It is easy to see this matches the tensor for the right hand side of Figure \ref{fig:SurgeryHeegaardDiagram}.
\end{proof}

\begin{cor}
Let $H$ be a finite-dimensional involutory Hopf algebra and let $M$ be a closed connected oriented 3-manifold, then
    \[Z_{\mathrm{Kup}}(M;H,\mu) = Z_{\mathrm{HKR}}(M; D(H),\mu,1)\]
\end{cor}

\begin{proof}
    Present $M$ as the manifold obtained by surgery on $S^3$ along some oriented framed link $L$. Then by Theorem \ref{intro:thm4}, Corollary \ref{cor:recoverKup}, and Theorem \ref{intro:thm2} we obtain the result.
\end{proof}

\section{State Sum Formulas}
\label{StateSum_Sec}
The construction outlined in Section \ref{DiagramInvariant_Sec} can be easily modified to obtain a family of invariants. Since the link components are disjoint, we could have chosen many different representations of $D(H)$ and mixed and matched across the link components. Choose some collection $\{(\rho_i, T_i)\}_{i\in I}$ of representations $\rho_i$ and trace-like functionals $T_i$ indexed over some (finite) indexing set $I$, then consider the sum (possibly with additional normalization)\[\sum_{\mathrm{col}:\{L_1,\dots ,L_m\}\to I} Z_{\mathrm{exK}}((M,L);H,\{(\rho_i,T_i)\}_{i\in \mathrm{Im}(\mathrm{col})})\] where $Z_{\mathrm{exK}}((M,L);H,\{(\rho_{\mathrm{col}(i)},T_i)\}_{i=1}^m)$ is the contraction of the tensor network obtained in the same way above except using $\rho_{\mathrm{col}(L_i)}$ for the tensor assigned to the $L_i$ component. This expression will be an invariant of the pair $(M,L)$ as usual, and in certain circumstances, recovers known invariants.

The following proposition gives a complete relationship between the extended Kuperberg invariant and the Witten-Reshetikhin-Turaev invariant.

\begin{prop}\label{prop:statesum}
    Let $(H,R)$ be a finite-dimensional involutory quasitriangular Hopf algebra and denote by $\rho_R$ the representation defined in Theorem \ref{intro:thm3}. Let $I$ be the set of irreducible representations of $H$. Then for each link $L = L_1\cup \cdots \cup L_m$ and each coloring $\mathrm{col}:\{1,\dots ,m\}\to I$, the Reshetikhin-Turaev colored link invariant $Z_{\mathrm{RT}}(L;\mathrm{col})$ (c.f. \cite{Reshetikhin1990}) is equal to the expression
    
    \[Z_{\mathrm{RT}}(L;\mathrm{col}) = Z_{\mathrm{exK}}((S^3,L);H,\{(\mathrm{col}(L_i)\circ \rho_R,\mathrm{Tr})\}_{i=1}^m)\]

    where $\mathrm{Tr}$ denotes the usual trace of linear endomorphisms on finite-dimensional vector spaces.
\end{prop}

\begin{proof}
    Let us first recall the Reshetikhin-Turaev (RT) colored link invariant. Present $L$ as the braid closure of some braid endowed with the blackboard framing. We may orient the link such that all strands in the braid are pointing down. Color the components by $\mathrm{col}$. Recall that the RT functor sends the cup and cap to the coevaluation and evaluation morphisms of $\mathrm{Rep}(H)$ and sends crossings to the braiding isomorphism (or its inverse). Since the RT colored link invariant has the properties that it is invariant under Reidemeister moves and that swapping the orientation of a link component and dualizing the color yields the same number, it is enough to only consider the case above.

    Now, the link is presented as the trace closure of a braid and as such the resulting image under the RT functor is the trace of the endomorphism defined by the image of only the braid under the same functor. Hence we need only analyze the tensor associated to a crossing since the traces on $\mathrm{Rep}(H)$ in both expressions match.

    
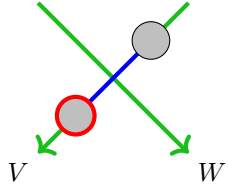
\begin{center}
        \begin{tikzpicture}
            \draw[->,darkgreen,ultra thick] (-1,1) -- (1,-1);
            \draw[darkgreen,ultra thick] (1,1) -- (0.5,0.5);
            \draw[darkgreen,->,ultra thick] (-0.5,-0.5) -- (-1,-1);
            \draw[blue,ultra thick] (0.5,0.5) -- (-0.5,-0.5);
            \draw[draw=black, fill=lightgray] (0.5,0.5) circle (0.25);
            \draw[draw=black, fill=lightgray] (-0.5,-0.5) circle (0.25);
            \draw[draw=red,ultra thick] (-0.5,-0.5) circle (0.25);
            \node[below left] at (-1,-1) {$V$};
            \node[below right] at (1,-1) {$W$};
        \end{tikzpicture}
    \end{center}

    Note that in the RT functor (if we use the same convention as in the HKR invariant above), this picture would be sent to $c_{W,V}^{-1}$, the inverse braiding, which is given by the action of the inverse R-matrix. Computing as in the proof of Theorem \ref{intro:thm3}, we obtain the following tensor:
    \begin{center}
        \begin{tikzpicture}
            \node at (0,0) (D) {$\Delta$};
            \node at (1,0) (M) {$M$};
            \node[darkgreen] at (2.5,0) (rhoR) {\fbox{$\rho_V\circ \rho_R$}};
            \node[darkgreen] at (-1.5,0) (rhoL) {\fbox{$\rho_W\circ \rho_R$}};
            \node at (0,-1) (e) {$e$};
            \node at (1,-1) (mu) {$\mu$};
            \draw[->] (D) -- (M);
            \draw[->] (D) -- (rhoL);
            \draw[<-] (M) -- (rhoR);
            \draw[->] (e) -- (D);
            \draw[->] (M) -- (mu);
            \draw[darkgreen,->] (rhoR) -- ++(1.5,0);
            \draw[darkgreen,->] (rhoL) -- ++(-1.5,0);
            \begin{scope}[shift = {(4,-2)}]
                \node at (-4.5,0) {$=$};
                \node[darkgreen] at (1.75,0) (rhoV) {\fbox{$\rho_V$}};
                \node[darkgreen] at (-2.75,0) (rhoW) {\fbox{$\rho_W$}};
                \node[darkgreen] at (0.5,0) (rhoR) {\fbox{$\rho_R$}};
                \node[darkgreen] at (-1.5,0) (rhoL) {\fbox{$\rho_R$}};
                \node at (-0.5,0) (S) {$S$};
                \draw[->] (rhoR) -- (S);
                \draw[->] (S) -- (rhoL);
                \draw[->,darkgreen] (rhoL) -- (rhoW);
                \draw[->,darkgreen] (rhoW) -- ++(-1,0);
                \draw[->,darkgreen] (rhoR) -- (rhoV);
                \draw[->,darkgreen] (rhoV) -- ++(1,0);
                \node at (3.5,0) {$\mu(e)$};
            \end{scope}
            \begin{scope}[shift = {(4,-4.5)}]
                \node at (-4.5,0) {$=$};
                \node[darkgreen] at (3,-0.5) (rhoV) {\fbox{$\rho_V$}};
                \node[darkgreen] at (-2.75,0) (rhoW) {\fbox{$\rho_W$}};
                \draw[->,darkgreen] (rhoW) -- ++(-1,0);
                \node at (1,0) (R) {$R$};
                \node at (0.7,0.2) {\tiny{$1$}};
                \node at (2,-0.5) (M) {$M$};
                \node at (1,-0.5) (i) {$i$};
                \node at (0,0) (S) {$S$};
                \draw[->] (R) -- (S); 
                \draw[->] (R) -- (M);
                \draw[<-] (M) -- (i);
                \draw[->] (M) -- (rhoV);
                \draw[->,darkgreen] (rhoV) -- ++(1,0);
            \begin{scope}[shift = {(-0.5,0.5)}]
                \node at (0,0) (R) {$R$};
                \node at (0,0.4) {\tiny{$1$}};
                \node at (-1,-0.5) (M) {$M^{\mathrm{op}}$};
                \node at (1,0.75) (eps) {$\epsilon$};
                \draw[->] (R) .. controls +(-0.5,0.5) and +(-1,0) .. (eps); 
                \draw[->] (R) -- (M);
                \draw[<-] (M) -- (S);
                \draw[->] (M) -- (rhoW);
            \end{scope}
            \end{scope}
            \begin{scope}[shift = {(4,-6)}]
                \node at (-4.5,0) {$=$};
                \node[darkgreen] at (0.75,0) (rhoV) {\fbox{$\rho_V$}};
                \node[darkgreen] at (-2.75,0) (rhoW) {\fbox{$\rho_W$}};
                \node at (-0.5,0) (R) {$R$};
                \node at (-0.8,0.2) {\tiny{$1$}};
                \node at (-1.5,0) (S) {$S$};
                \draw[->] (R) -- (S);
                \draw[->] (S) -- (rhoW);
                \draw[->] (R) -- (rhoV);
                \draw[->,darkgreen] (rhoW) -- ++(-1,0);
                \draw[->,darkgreen] (rhoV) -- ++(1,0);
            \end{scope}
        \end{tikzpicture}
    \end{center}

    as desired. In the same vein one can show that the (non-inverse) braiding is properly implemented as well.
\end{proof}

\vspace{0.5cm}
\noindent\textbf{Acknowledgements.} The second author would like to thank Greg Kuperberg for introducing the problem and providing critical insights to him. The authors thank Eric Samperton for helpful discussions. The authors are partially supported by National Science Foundation under Award
No. 2006667 and No. 2304990.

\vspace{0.5cm}
\noindent\textbf{Data availability statement.} All data relevant to this manuscript are included within the manuscript.

\bibliographystyle{alpha}

\end{document}